\documentclass[preprint,1p]{elsarticle}

\makeatletter
 \def\ps@pprintTitle{%
 	\let\@oddhead\@empty
 	\let\@evenhead\@empty
 	\def\@oddfoot{\footnotesize\itshape
 		{} \hfill\today}%
 	\let\@evenfoot\@oddfoot
 }
\makeatother
\usepackage[unicode]{hyperref}
\usepackage{latexsym}
\usepackage{indentfirst}
\usepackage{amsxtra}
\usepackage{amssymb}
\usepackage{amsthm}
\usepackage{amsmath}
\usepackage{mathrsfs} 
\usepackage{xcolor}
\usepackage{color}
\usepackage{amsfonts}

\usepackage[capitalise]{cleveref}

\newtheorem{theor}{Theorem}
\newtheorem*{theor*}{Theorem}
\newtheorem{prop}[theor]{Proposition}

\newtheorem{cor}[theor]{Corollary}
\newtheorem*{cor*}{Corollary}
\theoremstyle{definition}               %stile roman 
\newtheorem{defin}[theor]{Definition}
\newtheorem{ex}{Example}
\newtheorem{exs}[ex]{Examples}
\newtheorem{rem}{Remark}
\newtheorem{rems}[rem]{Remarks}

\DeclareMathOperator{\Sym}{Sym}
\DeclareMathOperator{\Aut}{Aut}
\DeclareMathOperator{\End}{End}
\DeclareMathOperator{\id}{id}
\DeclareMathOperator{\im}{Im}

\newcommand{\alphaa}[3]{\alpha^{#1}_{#2}{#3}}

\usepackage[makeroom]{cancel}
 
\usepackage{soul}
\usepackage{tikz}

\begin{document}

\begin{frontmatter}
	\title{Affine structures on groups and semi-braces\tnoteref{mytitlenote}}
	\tnotetext[mytitlenote]{This work is partially supported by the Dipartimento di Matematica e Fisica ``Ennio De Giorgi'' - Università del Salento. The author is partially supported by the ACROSS project ARS01\_00702 and is a member of GNSAGA (INdAM).}
	\author[unile]{Paola~STEFANELLI}
	\ead{paola.stefanelli@unisalento.it}
	%\cortext[c1]{Corresponding author}
	%\author[unile]{Marzia~MAZZOTTA}
	%\ead{marzia.mazzotta@unisalento.it}
	%\author[unile]{Paola~STEFANELLI}
	%\ead{paola.stefanelli@unisalento.it}
	\address[unile]{Dipartimento di Matematica e Fisica ``Ennio De Giorgi''
		\\
		Universit\`{a} del Salento\\
		Via Provinciale Lecce-Arnesano \\
		73100 Lecce (Italy)\\}

\begin{abstract} 
We introduce affine structures on groups and show they form a category  equivalent to that of semi-braces. In particular, such a new description of semi-braces includes that presented by Rump for braces. %\cite{Ru20-3}.
By specific affine structures, we  provide several instances of bi-skew braces, including some that are not $\lambda$-homomorphic.
Finally, we give a method for determining affine structures on the Zappa product of two groups both endowed with affine structures and prove that such a construction allows for obtaining semi-braces that are not matched product of semi-braces.
\end{abstract}

\end{frontmatter}

%*****************************************************
\section*{Introduction}
%*****************************************************

The Yang-Baxter equation is one of the fundamental equations of statistical mechanics that arose from two independent papers by Yang \cite{Ya67} and Baxter \cite{Ba72}. 
Motivated by its connection with many research areas, such as quantum groups, knot theory, and Hopf algebras, many mathematicians started to study this equation from different points of view.
A complete description of all solutions of the Yang-Baxter equation is still unknown, even if a large number of papers is dedicated to this topic. 
In this context, Drinfel'd \cite{Dr92} suggested the study of the special class of set-theoretic solutions. If $S$ is a set, a map  $r:S\times S \longrightarrow S\times S$ is said to be a \emph{set-theoretic solution of the Yang-Baxter equation}, or, briefly, a \emph{solution} on $S$, if the identity
\begin{align*}
    \left(r \times \id_S\right) 
    \left(\id_S \times r\right)
    \left(r \times \id_S\right) 
    = \left(\id_S \times r\right)
    \left(r \times \id_S\right)
    \left(\id_S \times r\right).
\end{align*}
is satisfied.
Writing $r(a,b) = (\lambda_a(b),\rho_b(a))$, with $\lambda_a$ and $\rho_b$ maps from $S$ into itself, for all $a,b \in S$,  $r$ is said to be \emph{left non-degenerate} if $\lambda_a$ is bijective, for every $a\in S$, \emph{right non-degenerate} if $\rho_b$ is bijective, for every $b\in S$, and \emph{non-degenerate} if $r$ is both left and right non-degenerate.
Bijective solutions have been initially studied by Etingof, Schedler and Soloviev \cite{ESS99}, Gateva-Ivanova and Van den Bergh \cite{GaVa98}, and Lu, Yan, and Zhu \cite{LuYZ00}. These pioneering works influenced and inspired the research from a more algebraic perspective in the next years.

In 2007, Rump discovered that any Jacobson radical ring determines a solution. This surprising fact motivated Rump to introduce \emph{braces}, algebraic structures including radical rings that turned out to be significant in the study of non-degenerate involutive solutions, namely, solutions $r$ such that $r^2=\id_{S\times S}$. 
Later on, algebraic structures close to braces, as skew braces and semi-braces, have been introduced for determining solutions. 
Following \cite{CCSt17} and \cite{JeAr19}, a set $B$ endowed with two binary operations $+$ and $\circ$ is said to be a \emph{(left) semi-brace} if $\left(B, +\right)$ is a semigroup, $\left(B, \circ\right)$ is a group, and the identity
\begin{align}\label{eq:semi}
    a\circ\left(b + c\right) = a\circ b + a\circ\left(a^- + c\right)\tag{$\ast$}
\end{align}
is satisfied, for all $a,b,c\in B$, where $a^-$ is the inverse of $a$ with respect to $\circ$. We highlight that in \eqref{eq:semi} it is implicit that the multiplication has higher precedence than the addition.
If $\left(B, +\right)$ is a left cancellative semigroup, we say that $B$ is a \emph{(left) cancellative semi-brace}.  %Denoted by $0$ the identity of the group $\left(B,\circ\right)$, we have that $0$ is an idempotent with respect to the sum. 
%In the cancellative case, it holds also that $0$ is a left identity of the semigroup $\left(B,+\right)$. 
As shown in \cite{CCSt17}, if the additive structure $\left(B,+\right)$ is a group, then \eqref{eq:semi} is equivalent to the relation
\begin{align}\label{eq:skew}
a\circ\left(b + c\right) = a\circ b - a + a\circ c,\tag{$\ast\ast$}
\end{align}
for all $a,b,c\in B$, 
namely, $B$ is a \emph{skew brace}, the structure introduced in \cite{GuVe17}.
%In this case, the identity of the group $\left(B, +\right)$ is exactly $0$. 
In addition, if the group $\left(B,+\right)$ is abelian, then $B$ is a \emph{brace}.\\
Any cancellative semi-brace $B$ gives rise to a left non-degenerate solution $r_B$ that is the map $r_B:B\times B\to B\times B$ defined by
$r_B\left(a,b\right) = \left(a\circ\left(a^- + b\right), \, \left(a^- + b\right)^-\circ b\right)$, for all $a,b\in B$. 
In particular, if $B$ is a skew brace, then $r_B$ is a non-degenerate bijective solution that is involutive if and only if $B$ is a brace.
Under mild assumptions, semi-braces produce degenerate solutions, among these the cubic ones, i.e., $r_B^3 = r_B$.
\\
Descriptions and classifications of these structures are object of investigation in many recent papers, as 
\cite{AcBo20, AlBy21, BaEs22,  CeJeOk20-2, CeSmVe19, Di21, GoNa21, JeKuVaVe19, Na19}, just to name a few.
In particular, in \cite{Ru19-2} Rump  provided a description of braces in terms of affine structures on groups. Using a different notation from that adopted by Rump, given a group $G = \left(B, \circ\right)$ we say that a map $\sigma:B\to\Sym_B$ is an \emph{affine structure} on $G$ if $\sigma$ is an anti-homomorphism from $G$ to the symmetric group on the set $B$ such that the following identity holds
$$
a\circ\sigma_a\left(b\right) = b\circ \sigma_b\left(a\right),
$$ 
for all $a,b\in B$. This notion allows to investigate isomorphism classes of braces having a fixed group as multiplicative group. 
The study of braces from the perspective of affine structures has led to new advances in the brace theory itself. 
More specifically, in \cite{Ru19-2} Rump provided constructions of finite braces by first proving that the determination of all affine structures can be reduced
to those of $p$,$q$-groups. 
Later on, in \cite{Ru20-3}, Rump obtained other instances of braces in terms of affine structures of decomposable solvable groups and, in \cite{Ru20-4}, proved what conjectured in \cite[Conjecture $6.9$]{GuVe17} by  
classifying all the affine structures of
a generalized quaternion group of order greater than or equal to $32$.

This paper aims to introduce a new notion of affine structure on a group and show how it is a useful tool for describing semi-braces and constructing special instances of them.
If $G = \left(B, \circ\right)$ is a group with identity $0$, we say that a map $\sigma:B\to B^B$ is an \emph{affine structure} on $G$ if $\sigma$ is a semigroup anti-homomorphism from the group $G$ to the set of all maps from $B$ into itself and the identity
$$
\sigma_{a}\left(b\circ \sigma_b\left(c\right)\right)
= \sigma_a\left(b\right)\circ \sigma_{\sigma_a\left(b\right)}\sigma_a\left(c\right)
$$
is satisfied, for all $a,b,c\in B$. 
Moreover, if $\sigma\left(B\right)\subseteq \Sym_B$, we say that $\sigma$ is a \emph{cancellative affine structure}. In such a case, if in addition $\sigma_a\left(0\right) = 0$, for every $a\in B$, we say that $\sigma$ is a \emph{groupal  affine structure}. We prove that Rump's affine structures are special groupal affine structures that we rename \emph{abelian affine structures} for our purposes.\\
We show the existence of a one-to-one correspondence  between affine structures of a given group $G$ and semi-braces having $G$ as multiplicative structure. More specifically, cancellative affine structures are in one-to-one correspondence with cancellative semi-braces and groupal affine structures with skew braces. In this correspondence, isomorphism classes of semi-braces correspond to equivalent classes of affine structures on $G$. In particular, we prove that semi-braces and affine structures are equivalent as categories. 

Part of this paper is dedicated to determine examples of affine structures and to obtain new ones starting from fixed affine structures. Particular attention has been posed to affine structures related to \emph{bi-skew braces}, namely, skew braces that are also skew braces reversing the roles of the multiplication and the sum.
Bi-skew braces was introduced and studied by Child in \cite{Ch19} %\cite[Definition 2]{Ch19}.
in the finite case and after they were investigated also in \cite{Ca20}, \cite{CarSt21}, and \cite{Ko21}.
Note that, recently, the authors in \cite{BaNeYa22}  focused on bi-skew braces $B$ that lie in the class of $\lambda$-homomorphic skew braces, i.e., skew braces for which $\lambda_{a+b} = \lambda_a\lambda_b$, where $\lambda_a\left(b\right) = - a + a\circ b$, for all $a,b\in B$.
%$\lambda_{a+b} = \lambda_a\lambda_b$.
Some of bi-skew braces we give in this work do not belong to such a class, hence,  as observed in \cite[Section 5]{BaNeYa22x}, our examples are different from that studied in \cite{Ch19}, \cite{Ko21, Ko22}, and \cite{CarSt21, CarSt21x}, that are all of $\lambda$-homomorphic type. 

Finally, starting from two affine structures $\sigma^S$ and $\sigma^T$ on two groups $H$ and $K$, we provide sufficient conditions so that specific maps $\alpha$ and $\beta$ make the map $\sigma^S\times \sigma^T$ into an affine structure on a group $H\bowtie_{\alpha,\beta}K$ that is isomorphic to the Zappa product of $H$ and $K$. In particular, such conditions are also necessary in the cancellative case. The multiplicative structure of any semi-brace obtained through such a method is the same obtained by means of the matched product of semi-braces contained in \cite[Theorem 6]{CCSt20-2} (see also \cite[p.249]{JeAr19}). However, specific examples show that, in general, semi-braces obtained by these two constructions are not the same, up to isomorphism. In addition, there exist maps $\alpha$ and $\beta$ that are suitable for obtaining semi-braces through both the techniques above mentioned. 
%
%
%In addition, given a skew brace $B$, by considering the opposite sum of $B$, we obtain another skew brace called the \emph{opposite skew brace of $B$} and denoted by $B^{op}$ (see \cite[Proposition 3.1]{KoTr20}).\\
%
%In \cite{GoNa21}, Gorshkov and Nasybullov proved that if $B$ is a minimal finite skew brace with solvable additive group and non-solvable multiplicative group, then the multiplicative group of $B$ is not simple.
%
%
%\textcolor{red}{\cite{Na19} Nasybullov: Connections between properties of the additive and the multiplicative groups of a two-sided skew brace.}
%
%
% \noindent In general, the additive structure of a cancellative semi-brace $B$ is a right group and $B = H + E$, where $E$ is the set of idempotents of $\left(B, +\right)$ and $H$ is the subgroup $B + 0$ of $\left(B, +\right)$. 
% %
% Specifically, $H$ is a skew brace and $E$ is a cancellative semi-brace such that $\left(E,+\right)$ is a right zero semigroup.
% %
% Observe that any group $\left(B,\circ\right)$ gives rise to a cancellative semi-brace by putting $a + b = b$, for all $a,b\in B$. This semi-brace is such that $H =\{0\}$ and $B = E$ and it is usually named the \emph{trivial semi-brace} on the group $\left(B,\circ\right)$. 
% %
% More broadly, instances of semi-brace can be obtained starting from a group $\left(B,\circ\right)$ by setting $a + b = b\circ f\left(a\right)$, for all $a,b\in B$,  where $f$ idempotent endomorphism of $\left(B, \circ\right)$. 
% \medskip
\medskip

%*****************************************************
\section{Affine structures on groups}\label{Section-aff}
%*****************************************************
This section aims to introduce affine structures on groups and clarify their connection with the notion already given by Rump in \cite{Ru19-2}. Moreover, we provide classes of examples of affine structures of different type.
\medskip

Hereinafter, whenever we consider a group $G = \left(B, \circ\right)$, we will denote its identity by $0$ and by $a^-$ the inverse of $a$, for every $a\in B$. 
In addition, denoted by $B^B$ the set of maps from $B$ into itself, for a map $\sigma:B\to B^B$, we write $\sigma_a$ to denote the image $\sigma\left(a\right)$, for every $a\in B$.
\medskip

Initially, let us recall the left-version of the notion of an affine structure originally introduced by Rump in \cite{Ru19-2} (see also \cite[p. 3]{Ru20-3}). 
We specify that we rename such a notion as an \emph{abelian affine structure} only for our purposes.
\begin{defin}
Let $G = \left(B, \circ\right)$ be a group and $\sigma:B\to\Sym_B$ an anti-homomorphism from $G$ to the symmetric group on the set $B$. Then, $\sigma$ is said to be an \emph{abelian affine structure} on $G$ if
\begin{align}\label{eq:abelian}
     a\circ\sigma_a\left(b\right)
     = b\circ \sigma_b\left(a\right),
\end{align}
for all $a,b\in B$.
\end{defin}

Below, we introduce the notion of the affine structure and illustrate the link with the previous definition.
\begin{defin}
Let $G = \left(B, \circ\right)$ be a group and $\sigma:B\to B^B$ a semigroup anti-homomorphism from the group $G$ to the set of all the maps from $B$ into itself. Then, we say that $\sigma$ is an \emph{affine structure} on $G$ if the identity
\begin{align}\label{eq:associative}
    \sigma_{a}\left(b\circ \sigma_b\left(c\right)\right)
    = \sigma_a\left(b\right)\circ \sigma_{\sigma_a\left(b\right)}\sigma_a\left(c\right)
\end{align}
holds, for all $a,b,c\in B$.
% In particular, if $\sigma:B\to\Sym_B$ is an anti-homomorphism from the group $G$ to the symmetric group on the set $B$, we say that $\sigma$ is a \emph{cancellative affine structure}. 
% In this case, if in addition, $\sigma_a\left(0\right) = 0$, for every $a\in B$, we say that $\sigma$ is a \emph{groupal  affine structure}.
In particular, if $\sigma\left(B\right)\subseteq \Sym_B$ (thus $\sigma$ is a group anti-homomorphism), we say that $\sigma$ is a \emph{cancellative affine structure}. 
In this case, if in addition, $\sigma_a\left(0\right) = 0$, for every $a\in B$, we say that $\sigma$ is a \emph{groupal  affine structure}.
\end{defin}
\medskip

Observe that, in general, any affine structure satisfies  $\sigma_0\left(0\right) = 0$. In fact, since $\sigma_0^2 = \sigma_0$, by \eqref{eq:associative}, we have that
$\sigma_0\left(0\right) = \sigma_0\left(0\circ \sigma_0\left(0\right)\right) = \sigma_0\left(0\right)\circ \sigma_{\sigma_0\left(0\right)}\sigma_0\left(0\right)$, hence $\sigma_{\sigma_0\left(0\right)}\left(0\right) = \sigma_{\sigma_0\left(0\right)}\sigma_0\left(0\right)= 0$. It follows that
$\sigma_0\left(0\right) 
%= 0\circ \sigma_0\left(0\right)
= \sigma_{\sigma_0\left(0\right)}\left(0\right)\circ \sigma_{\sigma_{\sigma_0\left(0\right)}\left(0\right)}\sigma_{\sigma_0\left(0\right)}\left(0\right)$ from which we obtain that $\sigma_0\left(0\right) = \sigma_{\sigma_0\left(0\right)}\left(0\circ\sigma_0\left(0\right)\right) = \sigma_{\sigma_0\left(0\right)}\left(0\right) = 0$.
\medskip

Let us note that abelian affine structures are special cases of groupal affine structures. 
Indeed, if $\sigma$ is an abelian affine structure on a group $\left(B, \circ\right)$ and $a,b,c\in B$,  we have that
\begin{align*}
    a\circ \sigma_{a}\left(b\circ \sigma_b\left(c\right)\right)
    &= b\circ \sigma_b\left(c\right)\circ \sigma_{b\circ \sigma_b\left(c\right)}\left(a\right)&\mbox{by \eqref{eq:abelian}}\\
    &= b\circ \sigma_b\left(c\right)\circ
    \sigma_{\sigma_b\left(c\right)}\sigma_b\left(a\right)\\
    &= b\circ \sigma_b\left(a\right)\circ
    \sigma_{\sigma_b\left(a\right)}\sigma_b\left(c\right)&\mbox{by \eqref{eq:abelian}}\\
    &= b\circ \sigma_b\left(a\right)\circ
    \sigma_{b\circ \sigma_b\left(a\right)}\left(c\right)\\
    &= a\circ \sigma_a\left(b\right)\circ
    \sigma_{a\circ \sigma_a\left(b\right)}\left(c\right)&\mbox{by \eqref{eq:abelian}}\\
    &= a\circ \sigma_a\left(b\right)\circ
    \sigma_{\sigma_a\left(b\right)}\sigma_a\left(c\right),
\end{align*}
thus \eqref{eq:associative} holds. Moreover, 
$\sigma_a\left(0\right) 
    = a^-\circ a\circ\sigma_a\left(0\right)
    = a^-\circ 0\circ\sigma_0\left(a\right)
= a^-\circ a = 0$.
\medskip

In addition, observe that there do not exist intermediate affine structures between the general case and the cancellative one, since any affine structure satisfying $\sigma_0 = \id_B$ is cancellative.
\medskip

\begin{ex}
    Let $G = \left(B,\circ\right)$ be a group.  If $|B| > 1$,  then the map $\sigma:B\to\Sym_B$ given by $\sigma_a\left(b\right) = a^-\circ b$, for all $a,b\in B$, is an anti-homomorphism that satisfies \eqref{eq:associative} and, if $a\neq 0$, then $  \sigma_a\left(0\right) = a^-\neq 0$. Hence, $\sigma$ is an affine structure that is cancellative but not groupal.
\end{ex}
\medskip

%\noindent Easy instances of affine structures can be obtained as shown below.
\begin{ex}
Let $G = \left(B,\circ\right)$ be a group and $f$ an idempotent endomorphism of $G$. Then,  the map $\sigma:B\to B^B$ given by $\sigma_a = f$, for every $a\in B$, is an affine structure on $G$ that is not cancellative if $f\neq\id_B$. Clearly, if $f =\id_B$, then $\sigma$ is groupal and it is abelian if and only if the group $G$ is abelian.
\end{ex}
\medskip

%Bi-skew braces
\begin{ex}\label{exs:1}
Let $G = \left(B,\circ\right)$ be a group and $f$ an idempotent endomorphism of $G$. 
Then, the map $\sigma:B\to\Sym_B$ defined by $\sigma_a\left(b\right) = f\left(a\right)^-\circ b\circ f\left(a\right)$, for all $a,b\in B$,  is a groupal affine structure on $G$. Note that, if $f = \id_B$, then $\sigma$ is abelian if and only $G$ is an abelian group.
% \item the map $\sigma:B\to\Sym_B$ defined by $\sigma_a\left(b\right) = a^-\circ b\circ a$, for every $a,b\in B$ is an affine structure that is abelian if and only if the group $\left(B,\circ\right)$ is abelian.
\end{ex}
\medskip

%Bi-skew brace
From now on, whenever we consider a cyclic group $G$, we assume that $G = \langle\, g\, \rangle$, unless otherwise stated; in addition, if $G$ is finite of order $m$, we will refer to $G$ by denoting it as $C_m$.
\begin{exs}\label{ex:Z}
Let $G = \left(B,\circ\right)$ be the infinite cyclic group. %with $B:= \{g^k \ | \ k\in\mathbb{Z}\}$.
\begin{enumerate}
\item The map $\sigma:B\to\Sym_{B}$ defined by
    $\sigma_{g^k}\left(g^l\right):= g^{\left(-1\right)^{k}l}$, for all $k,l\in \mathbb{Z}$, is a groupal affine structure of $G$. Moreover, $\sigma$ is not abelian, since, for instance, 
    $g\circ\sigma_g\left(g^2\right) = g^{-1}$, instead $g^2\circ\sigma_{g^2}\left(g\right) = g^3$.
\item The map $\sigma:B\to\Sym_{B}$ defined by
    $\sigma_{g^k}\left(g^l\right):= g^{k\left(-1 + \left(-1\right)^l\right) + l}$, for all $k,l\in \mathbb{Z}$ is a groupal affine structure on $G$ that is not abelian. Indeed, one can check that the maps $\sigma_{g^k}$ are bijections on $B$, for every $k\in\mathbb{Z}$, and 
    that $\sigma$ is an anti-homomorphism of groups.
    Note also that $\sigma_{g^k}\left(0\right) = 0$.
    Furthermore, if $k,l,m\in\mathbb{Z}$, then
    $\left(-1\right)^{l\left(-1 + \left(-1\right)^m\right) + m} = \left(-1\right)^m$,
    % and
    % \begin{align*}
    %     &k\left(-1+\left(-1\right)^l\right)+l + \left(k\left(-1+\left(-1\right)^l\right)+l\right)\left(1 + \left(-1\right)^m\right) + l\left(-1+\left(-1\right)^m\right)+m\\
    %     &\quad = k\left(- 1 + \left(-1\right)^{l\left(-1\right)^m + m}\right)
    %     + l\left(-1\right)^m + m,
    % \end{align*}
    hence we obtain that 
    \begin{align*}
        \sigma_{g^k}\left(g^l\right)
        \circ\, &\sigma_{\sigma_{g^k}\left(g^l\right)}\sigma_{g^l}\left(g^m\right)
        = g^{k\left(-1+\left(-1\right)^l\right)+l}
        \circ
        \sigma_{g^{k\left(-1+\left(-1\right)^l\right)+l}}
        \left(g^{l\left(-1+\left(-1\right)^m\right)+m}\right)\\
        &= g^{k\left(-1+\left(-1\right)^l\right)+l + \left(k\left(-1+\left(-1\right)^l\right)+l\right)\left(1 + \left(-1\right)^m\right) + l\left(-1+\left(-1\right)^m\right)+m}\\
        &= g^{k\left(- 1 + \left(-1\right)^{l\left(-1\right)^m + m}\right)
        + l\left(-1\right)^m + m}\\
        &= \sigma_{g^k}\left(g^{l\left(-1\right)^m + m}\right)
        = \sigma_{g^k}\left(g^{l + l\left(-1+ \left(-1\right)^m\right) + m}\right)\\
        &= \sigma_{g^k}\left(g^l\circ \sigma_{g^l}\left(g^m\right)\right),
    \end{align*}
    i.e., \eqref{eq:associative} holds. 
    Finally, $\sigma$ is not abelian since $g\circ\sigma_g\left(g^2\right) = g^3$ and  $g^2\circ \sigma_{g^2}\left(g\right) = g^{-1}$.
\end{enumerate}
\end{exs}
\medskip

\begin{ex}\label{ex:Z/Z2l}
    %\label{ex:2}
    %\label{ex:ex-Z-opp}
    In the finite case, one can similarly define the maps $\sigma$ in $1.$ and $2.$ of \cref{ex:Z} on the cyclic group
    $G = C_m$ 
    with $m$ even, and obtain that they are affine structures that are not abelian if $m\nmid 4$.
    % Il gruppo ottenuto è il gruppo diedrale di ordine $m = 2l$, per un certo $l$.
\end{ex}
\medskip

The following proposition illustrates how to define affine structures on isomorphic groups with each other.
\begin{prop}\label{iso-aff-str}
    Let $\sigma$ be an affine structure on a group $G = \left(B,\circ\right)$. If $H = \left(C,\circ\right)$ is a group and $f:B\to C$ a group isomorphism  from $G$ to $H$, then the map $\varphi:C\to C^C$ given by 
    $$
    \varphi_u:= f\sigma_{f^{-1}\left(u\right)}f^{-1},
    $$
    for every $u\in C$, is an affine structure on $H$. 
    Moreover, if $\sigma$ is cancellative, then $\varphi$ is cancellative; if in addition $\sigma$ is groupal, then $\varphi$ is groupal.
    \begin{proof}
        Clearly, $\varphi$ is a semigroup anti-homomorphism and if $u,v,w\in C$, we get
        \begin{align*}
            \varphi_u\left(v\circ\varphi_v\left(w\right)\right)
            &= f\sigma_{f^{-1}\left(u\right)}\left(f^{-1}\left(v\right)\circ\sigma_{f^{-1}\left(v\right)} f^{-1}\left(w\right)\right)\\
            &= f\left(\sigma_{f^{-1}\left(u\right)}f^{-1}\left(v\right)\circ \sigma_{\sigma_{f^{-1}\left(u\right)}f^{-1}\left(v\right)}\sigma_{f^{-1}\left(u\right)}f^{-1}\left(w\right)\right)&\mbox{by \eqref{eq:associative}}\\
            &=\varphi_{u}\left(v\right)\circ
            f\sigma_{f^{-1}\varphi_{u}\left(v\right)}\sigma_{f^{-1}\left(u\right)}f^{-1}\left(w\right)\\
            &= \varphi_{u}\left(v\right)\circ
            f\sigma_{f^{-1}\varphi_{u}\left(v\right)}f^{-1}\varphi_{u}\left(w\right)\\
            &= \varphi_{u}\left(v\right)\circ\varphi_{ \varphi_{u}\left(v\right)} \varphi_{u}\left(w\right),
        \end{align*}
        i.e., \eqref{eq:associative} holds. 
        % Moreover, if $\sigma$ is cancellative, then $\varphi$ is clearly a group anti-homomorphism and if, in addition $\sigma$ is groupal, we clearly have that $\varphi_u\left(0\right) = 0$, for every $u\in C$.
        % Therefore, the claim follows.
        Finally, the remaining part of the proof is trivial.
    \end{proof}
\end{prop}
\medskip

\begin{defin}
Let $G = \left(B,\circ\right), H = \left(C,\circ\right)$ be groups and $\sigma,\varphi$ affine structures on $G$ and $H$, respectively. If there exists a group homomorphism $f$ from $G$ to $H$ such that
\begin{align*}
\varphi_{f\left(a\right)}f = f\sigma_{a},
\end{align*}
for every $a\in B$, we say that $\sigma$ and $\varphi$ are \emph{homomorphic via $f$} and $f$ is a \emph{homomorphism between $\sigma$ and $\varphi$}. If in addition $f$ is bijective, we say that $\sigma$ and $\varphi$ are \emph{equivalent via $f$} and $f$ is an \emph{isomorphism between $\sigma$ and $\varphi$}.
\end{defin}

\medskip
% ****
% \begin{rem}
% In light of \cref{ex:C8+C8=D8}, given an affine structure $\sigma$ on a group $G$, \cref{prop:comp-aff} can allow one to produce a new affine structure on the same group $G$ that is not isomorphic to $\sigma$.
% \end{rem}
% ****

Given a group $G$, the following result provides sufficient conditions to obtain a new affine structure on $G$ starting from two given ones on the same group $G$. 
The subsequent \cref{ex:C8+C8=D8} will show that, given an affine structure $\sigma$ on a group $G$, \cref{prop:comp-aff} turns out to be helpful to produce affine structures on $G$ that are not isomorphic to $\sigma$. 
\begin{theor}\label{prop:comp-aff}
  Let $G = \left(B, \circ\right)$ be a group and $\varphi,\omega:B\to B^B$ affine structures on $G$. If the following conditions are satisfied
  \begin{align}%\label{cond:om-var}
   \varphi_a\omega_b \label{cond:om-var-com}
    &= \omega_b\varphi_a\tag{$c1$}\\
    \varphi_{b\circ \omega_a\left(b\right)^-}\label{cond:om-var}
    &= \omega_{\varphi_{a}\omega_a\left(b\right)\circ \omega_a\left(b\right)^-},\tag{$c2$}
  \end{align}
  for all $a,b\in B$, then the map $\sigma:B\to B^B$ given by $\sigma_a:= \varphi_a\omega_a$, for every $a\in B$, is an affine structure on $G$. 
 \begin{proof}
     Initially, note that $\sigma$ is a semigroup anti-homomorphism.
     Besides, if $a,b,c\in B$,
     \begin{align*}
       \sigma_a\left(b\circ\sigma_b\left(c\right)\right)
       %&= \varphi_a\omega_a\left(b\circ\omega_b\varphi_b\left(c\right)\right)\\
       &= \varphi_a\left(\omega_a\left(b\right)\circ\omega_{\omega_a\left(b\right)}\omega_a\varphi_b\left(c\right)\right)&\mbox{$\omega$ affine structure}\\
       &= \varphi_a\left(\omega_a\left(b\right)\circ\omega_{\omega_a\left(b\right)}\omega_a\varphi_0\varphi_b\left(c\right)\right)\\
       &= \varphi_a\left(\omega_a\left(b\right)\circ
       \varphi_{\omega_a\left(b\right)}\varphi_{\omega_a\left(b\right)^-}
       \omega_{\omega_a\left(b\right)}\omega_a\varphi_b\left(c\right)\right)&\mbox{by \eqref{cond:om-var-com}}\\
       &= \varphi_a\omega_a\left(b\right)\circ
       \varphi_{\varphi_a\omega_a\left(b\right)}
       \varphi_a\varphi_{\omega_a\left(b\right)^-}\omega_{\omega_a\left(b\right)}\omega_a\varphi_b\left(c\right)&\mbox{$\varphi$ affine structure}\\
       &= \varphi_a\omega_a\left(b\right)\circ
       \varphi_{\varphi_a\omega_a\left(b\right)}
       \varphi_a\varphi_{\omega_a\left(b\right)^-}\omega_{\omega_a\left(b\right)}\varphi_b\omega_a\left(c\right)
     \end{align*}
and
\begin{align*}
    \sigma_a\left(b\right)\circ\sigma_{\sigma_a\left(b\right)}\sigma_a\left(c\right)
    &= \varphi_a\omega_a\left(b\right)\circ
    \varphi_{\varphi_a\omega_a\left(b\right)}\omega_{\varphi_a\omega_a\left(b\right)}\varphi_a\omega_a\left(c\right)\\
    &= \varphi_a\omega_a\left(b\right)\circ
    \varphi_{\varphi_a\omega_a\left(b\right)}
    \varphi_a\omega_{\varphi_a\omega_a\left(b\right)}\omega_a\left(c\right).
\end{align*}
Observing that
\begin{align*}
    \varphi_{\omega_a\left(b\right)^-}\omega_{\omega_a\left(b\right)}\varphi_b
    &= \omega_{\omega_a\left(b\right)}\varphi_{\omega_a\left(b\right)^-}\varphi_b
    = \omega_{\omega_a\left(b\right)}\varphi_{b\circ \omega_a\left(b\right)^-}\\
    &= \omega_{\omega_a\left(b\right)}
    \omega_{\varphi_a\omega_a\left(b\right)\circ\omega_a\left(b\right)^-}&\mbox{by \eqref{cond:om-var}}\\
    &= \omega_{\varphi_a\omega_a\left(b\right)\circ 0}
    =  \omega_{\varphi_a\omega_a\left(b\right)},
\end{align*}
we obtain that \eqref{eq:associative} is satisfied. 
Therefore, the claim follows.
\end{proof}
\end{theor}

\begin{cor}\label{cor-prop:comp-aff}
Let $G = \left(B, \circ\right)$ be a group and $\varphi,\omega:B\to B^B$ cancellative affine structures on $G$. If the following conditions are satisfied 
  \begin{align}
  \varphi_a\omega_b&= \omega_b\varphi_a\tag{$c1$}\\
  \varphi_{\omega_{a^-}\left(b\right)\circ b^-}
  &= \omega_{\varphi_{a}\left(b\right)\circ b^-}\label{cond:om-var-2},\tag{$c2'$}
  \end{align}
  for all $a,b\in B$, then the map $\sigma:B\to B^B$ given by $\sigma_a:= \varphi_a\omega_a$, for every $a\in B$, is a cancellative affine structure on $G$. Moreover, if $\omega$ and $\varphi$ are groupal, then $\sigma$ is groupal.
\end{cor}
\begin{proof}
It is a routine computation to verify that \eqref{cond:om-var-2} is equivalent to \eqref{cond:om-var}.
Besides, note that if $\omega$ and $\varphi$ are cancellative, then $\sigma$ is a group anti-homomorphism. Finally, if $\omega$ and $\varphi$ are groupal, it clearly holds that $\sigma_a\left(0\right) = 0$, for every $a\in B$.
\end{proof}
\medskip

The following is an example of an abelian affine structure constructed as in \cref{cor-prop:comp-aff} and starting from one that is groupal but not abelian. 
\begin{ex}\label{ex:C8+C8=D8}
Let $G = \left(B,\circ\right)$ be the cyclic group %$\langle\, g\,\rangle$ of order $8$ 
$C_8$
and consider $\omega,\varphi: B\to\Sym_B$ both equal to the groupal affine structure as in \cref{ex:Z}-$2.$, namely,  $\omega_{g^k}\left(g^l\right) = g^{k\left(-1 + \left(-1\right)^l\right) + l}$, for all $k,l\in\mathbb{Z}$. Then, the assumptions of \cref{cor-prop:comp-aff} are satisfied.
Indeed, if $k,h,l\in\mathbb{Z}$, we clearly obtain that $\omega_{g^k}\varphi_{g^h} = \varphi_{g^h}\omega_{g^k}$; in addition setting  $u:=k + k\left(-1\right)^{l+1} + k\left(-1\right)^{h+1} + k\left(-1\right)^{h + l}$ and $t:=- k + k\left(-1\right)^l + k\left(-1\right)^h + k\left(-1\right)^{h+l+1}$,
\begin{align*}
 \omega_{\varphi_{g^k}\left(g^h\right)}\left(g^l\right) 
 &= \sigma_{g^{k\left(-1 + \left(-1\right)^h\right)+h}}\left(g^l\right)
 = g^{\left(-k + k\left(-1\right)^h +h\right)\left(- 1 + \left(-1\right)^l\right)+ l}
 = g^{u - h + h\left(-1\right)^{l}+ l}
\end{align*}
and, since $\left(g^k\right)^- = g^{-k}$, 
\begin{align*}
    \varphi_{\omega_{\left(g^k\right)^-}\left(g^h\right)}\left(g^l\right)
    &= \sigma_{g^{-k\left(-1 + \left(-1\right)^h\right)+h}}\left(g^l\right)
    = g^{\left(k + k\left(-1\right)^{h+1} + h\right)\left(-1 + \left(-1\right)^l\right) + l}
    %&=g^{- k + k\left(-1\right)^l + k\left(-1\right)^{h+2} + k\left(-1\right)^{h+l+1} -h + h\left(-1\right)^l + l}\\
    = g^{t - h + h\left(-1\right)^l + l}.
\end{align*}  
Note that if $h,l$ are odd, then $u = 4k$ and $t = -4k$, thus \eqref{cond:om-var} holds. Moreover, in all the the other cases  $u = t = 0$, hence \eqref{cond:om-var} is always satisfied. Therefore, the map $\sigma:B\to \Sym_B$ given by $\sigma_{g^k}:= \omega_{g^k}^2$ is an affine structure on the group $G$.
Finally, by distinguishing all the cases in which $k,l$ are even  or odd,  it is a routine computation to check that  \eqref{eq:abelian}  holds, namely, $\sigma$ is abelian.
\end{ex}
\medskip

%*****************************************************
\section{Semi-braces and affine structures}\label{section-semi-aff}
%*****************************************************
In this section, we prove there exists a one-to-one correspondence  between affine structures on a given group $G$ and semi-braces having $G$ as multiplicative structure. 
In particular, we show that semi-braces and affine structures are equivalent (or isomorphic) as categories.
\medskip

Initially, let us, briefly, recall some basics on the structure of the semi-brace.
\begin{defin}[\cite{CCSt17}, \cite{JeAr19}]
A set $B$ endowed with two binary operations $+$ and $\circ$ is said to be a \emph{(left) semi-brace} if $\left(B, +\right)$ is a semigroup, $\left(B, \circ\right)$ is a group, and the identity
\begin{align*}
    a\circ\left(b + c\right) = a\circ b + a\circ\left(a^- + c\right)\tag{$\ast$}
\end{align*}
holds, for all $a,b,c\in B$, where $a^-$ denotes the inverse of $a$ with respect to $\circ$. 
Moreover, if $\left(B, +\right)$ is a (left) cancellative semigroup, then $B$ is said to be a \emph{left cancellative semi-brace}.
The structures $\left(B, +\right)$ and $\left(B, \circ\right)$ are named the \emph{additive} and the \emph{multiplicative structure} of $B$, respectively. 
\end{defin}

\noindent Denoted by $0$ the identity of the group $\left(B,\circ\right)$, we have that $0$ is an idempotent with respect to the sum. In the cancellative case, it holds also that $0$ is a left identity of the semigroup $\left(B,+\right)$. 
%As shown in \cite{CCSt17}, if the additive structure $\left(B,+\right)$ is a group, then \eqref{eq:semi} is equivalent to the relation
%\begin{align*}
%a\circ\left(b + c\right) = a\circ b - a + a\circ c,\tag{$\ast\ast$}
%\end{align*}
%for all $a,b,c\in B$, 
If the additive structure is a group, then $B$ is a \emph{skew brace}, the structure introduced in \cite{GuVe17}.
In this case, the identity of the group $\left(B, +\right)$ is exactly $0$. Furthermore, if $\left(B,+\right)$ is an abelian group, then $B$ is a \emph{brace}.
\medskip

Before illustrating the close link between affine structures and semi-braces, let us recall two fundamental maps in the study of semi-braces and of the solutions associated to them, namely, the maps  
\begin{align*}
\lambda_a:B\to B,\, b\mapsto a\circ\left(a^- + b\right)
\qquad
 \rho_b:B\to B,\, a\mapsto \left(a^- + b\right)^-\circ b,
\end{align*}
for all $a,b\in B$. 
In the following proposition, we collect some properties that are useful for our aims. For a fuller treatment, we refer the reader to \cite{GuVe17}, \cite{CCSt17}, \cite{JeAr19}, and \cite{CaCeSt22}.
\begin{prop}\label{prop:lambda-rho}
    Let $B$ be a semi-brace. Then, the following statements hold:
    \begin{enumerate}
        \item $\lambda_a$ is an endomorphism of $\left(B,+\right)$, for every $a\in B$;
        \item the map $\lambda: B\to \End(B,+), \,a\mapsto\lambda_a$ is a homomorphism from the group $(B, \circ)$ to the monoid of the endomorphisms of $\left(B,+\right)$;
        \item if $B$ is left cancellative, then $\lambda_a$ is bijective, for every $a\in B$;
        %\item if $g\in G$ and $b\in B$, then $\lambda_g\left(b\right) = - g + g\circ b$;
        \item if $B$ is left cancellative the map $\rho: B\to B^B, \,b\mapsto\rho_b$ is a semigroup anti-homomorphism from $(B, \circ)$  into the monoid of the maps from $B$ into itself;
        \item if $B$ is a skew brace, then $\rho_b$ is bijective, for every $b\in B$.
    \end{enumerate}
\end{prop}
\medskip

\begin{theor}\label{th:aff-skew}
Let $G = \left(B, \circ\right)$ be a group and $\sigma$ an affine structure on $G$. Then, the binary operation given by
\begin{align*}
    a + b:= a\circ\sigma_a\left(b\right),
\end{align*}
for all $a,b\in B$, makes $B$ into a semi-brace that we call the \emph{semi-brace associated to $\sigma$} and denote by $B_{\sigma}$. 
Moreover, if $\sigma$ is cancellative then $B_{\sigma}$ is cancellative and we call it the \emph{cancellative semi-brace associated to $\sigma$}. If in addition $\sigma$ is groupal, then $B_\sigma$ is a skew brace which we call the \emph{skew brace associated to $\sigma$}.
\begin{proof}
% If $a,b,c\in B$, we get
% \begin{align*}
%     \left(a + b\right) + c
%     = a\circ \sigma_a\left(b\right) + c
%     = a\circ\sigma_a\left(b\right)\circ \sigma_{\sigma_a\left(b\right)}\sigma_a\left(c\right)
%     = a\circ\sigma_a\left(b\circ\sigma_b\left(c\right)\right)
%     = a + \left(b + c\right),
% \end{align*}
At first, note that, by \eqref{eq:associative}, we easily obtain that $\circ$ is associative, hence $\left(B, \circ\right)$ is a semigroup. 
Moreover,
\begin{align*}
    a\circ b + a\circ\left(a^- + c\right)
    &= a\circ b + \sigma_{a^-}\left(c\right)
    =  a\circ b\circ \sigma_{a\circ b}\sigma_{a^-}\left(c\right)
    =  a\circ b\circ\sigma_b\left(c\right)
    = a\circ\left(b + c\right),
\end{align*}
thus \eqref{eq:semi} is satisfied.
Now, if we assume that $\sigma$ is cancellative and $a + b = a + c$, then $a\circ\sigma_a\left(b\right) = a\circ\sigma_a\left(c\right)$ that implies $b = c$, hence $\left(B, +\right)$ is a cancellative semigroup.
Furthermore, if we suppose that $\sigma$ is groupal, we obtain that 
$a + 0 = a\circ\sigma_a\left(0\right) = a$ and, since $\sigma_0=\id_B$, $0 + a = a$. 
Finally, $a + \sigma_{a^-}\left(a^-\right)
    = a\circ\sigma_a\sigma_{a^-}\left(a^-\right)
    = a\circ a^- = 0$ and
\begin{align*}
    \sigma_{a^-}\left(a^-\right) + a
    = \sigma_{a^-}\left(a^-\right)\circ\sigma_{\sigma_{a^-}\left(a^-\right)}\left(a\right)
    = \sigma_{a^-}\left(a^-\circ\sigma_{a^-}\sigma_a\left(a\right)\right) 
    = \sigma_{a^-}\left(0\right) = 0,
\end{align*}
namely, $a$ has opposite $-a = \sigma_{a^-}\left(a^-\right)$.
Therefore the statement is proved.
\end{proof}
\end{theor}
\medskip

\begin{rem}
If $B_\sigma$ is the semi-brace associated to an affine structure $\sigma$, by \eqref{eq:associative} we have that
\begin{align*}
    \sigma_a\left(b+c\right)
    = \sigma_a\left(b\circ\sigma_b\left(c\right)\right)
    = \sigma_a\left(b\right)\circ\sigma_{\sigma_a\left(b\right)}\sigma_a\left(c\right)
    = \sigma_a\left(b\right) + \sigma_a\left(c\right),
\end{align*}
for all $a,b, c\in B$, i.e., $\sigma_a\in\End\left(B, +\right)$, for every $a\in B$.
\end{rem}
\medskip

\begin{theor}\label{th:sk-aff}
Let $B$ be a semi-brace. Then, the map $\sigma:B\to B^B, a\mapsto \lambda_{a^-}$ is an affine structure of $\left(B, \circ\right)$. Moreover, $a + b = a\circ\sigma_a\left(b\right)$, for all $a,b\in B$, namely the semi-brace $B_{\sigma}$ coincides with $B$.
In addition, cancellative affine structures correspond to cancellative semi-braces, and  groupal affine structures correspond to skew braces.
\begin{proof}
Clearly, by \cref{prop:lambda-rho}, we have that $\sigma$ is an anti-homomorphism from the semigroup $\left(B, \circ\right)$ to the monoid $B^B$. If $a,b,c\in B$, we have that 
\begin{align*}
 \sigma_{a}\left(b\circ \sigma_b\left(c\right)\right)
 &= \lambda_{a^-}\left(b + c\right)
 = \lambda_{a^-}\left(b\right) + \lambda_{a^-}\left(c\right)
 =\lambda_{a^-}\left(b\right)\circ\sigma_{\lambda_{a^-}\left(b\right)}\lambda_{a^-}\left(c\right)\\
 &= \sigma_a\left(b\right)\circ \sigma_{\sigma_a\left(b\right)}\sigma_a\left(c\right),  
\end{align*}
i.e., \eqref{eq:associative} is satisfied, and $a\circ\sigma_a\left(b\right) = a\circ\lambda_{a^-}\left(b\right) = a + b$. 
Furthermore, we trivially have that the sum defined as in \cref{th:aff-skew} coincides with the sum of the semi-brace $B$, hence $B_{\sigma}$ coincides with the semi-brace $B$. 
Finally, the remaining part of the proof is trivial.
\end{proof}
\end{theor}

\medskip

\noindent From now on, given a semi-brace $B$, let us call the map $\sigma$ as in \cref{th:sk-aff} the \emph{affine structure associated to $B$} and denote it by $\sigma^{B}$.
\begin{rem}\label{rem:aff}
If $\sigma$ is an affine structure on a group $G = \left(B,\circ\right)$ and $B_\sigma$ the semi-brace associated to $\sigma$, then we clearly have that $\sigma^B = \sigma$.
%In fact, if $a,b\in B$, then 
%$\sigma_a\left(b\right) = \lambda_{a^-}\left(b\right) = \sigma^B_a\left(b\right)$.

\end{rem}

\bigskip

As a consequence of \cref{th:aff-skew} and \cref{th:sk-aff}, one can define a bijective correspondence between the class of semi-braces having a fixed group $G$ as multiplicative group and that of affine structures on $G$ itself.
\begin{cor}
Let $G =\left(B,\circ\right)$ be a group, $\mathcal{S}$ the class of semi-braces having $G$ as multiplicative group, and  $\mathcal{A}$ the class of affine structures on $G$. Then, the map  defined by 
\begin{align*}
   \psi:\mathcal{S}\to\mathcal{A}, \
   B\mapsto \sigma^B
\end{align*}
is a bijection.
%\\*** In this correspondence, isomorphism classes of skew braces correspond to equivalent classes of affine structure on $G$.***
% \begin{proof}
% Let $B$ and $C$ be skew braces both having $G$ as multiplicative group and such that  $\psi\left(B\right) = \psi\left(C\right)$. Then, by \cref{th:sk-aff} we have that the skew braces $B$ and $C$ coincide, thus $\psi$ is injective. Moreover, if $\sigma$ is an affine structure on $G$, by \cref{th:aff-skew} we obtain a structure of skew brace $B_\sigma$ on $B$ and, by \cref{th:sk-aff},  $\sigma^B = \sigma$, thus $\psi$ is surjective.
% \end{proof}
\end{cor}

\medskip 

The following theorem shows a linkage between homomorphic classes of semi-braces and homomorphic  classes of affine structures.
\begin{theor}\label{th:aff_semi-braces}
If $B, C$ are semi-braces and $f$ is a homomorphism between $B$ and $C$, then the affine structures $\sigma^B$ and $\sigma^C$ are homomorphic via $f$.\\
Conversely, if $\sigma, \varphi$ are affine structures on the groups $G = \left(B, \circ\right)$ and $H = \left(C, \circ\right)$, respectively, and $f$ is a homomorphism between $\sigma$ and $\varphi$, then $f$ is a homomorphism from the semi-brace $B_\sigma$ to $C_\varphi$.   
\begin{proof}
Initially, assume that $B, C$ are semi-braces and $f$ is an homomorphism between $B$ and $C$. Then, if $a,b\in B$, we have that
 \begin{align*}
    f\sigma^B_{a}\left(b\right)
    %= f\lambda_{a^-}\left(b\right)
     = f\left(a^-\circ\left(a + b\right)\right)
    = f\left(a\right)^-\circ\left(f\left(a\right) +        f\left(b\right)\right)
     %= \lambda_{f\left(a\right)^-}f\left(b\right)
     = \sigma^C_{f\left(a\right)}f\left(b\right),
\end{align*}
i.e., $\sigma^B$ and $\sigma^C$ are homomorphic via $f$.\\
Conversely, let $\sigma,\varphi$ affine structures on $G$ and $H$, respectively, and assume that they are homomorphic via a group homomorphism $f$ from $G$ to $H$, thus $\varphi_{f\left(a\right)}f = f\sigma_{a}$, for every $a\in B$. If $a,b\in B$, it follows that
\begin{align*}
    f\left(a+b\right)
    = f\left(a\circ \sigma_a\left(b\right)\right)
    = f\left(a\right)\circ f\sigma_a\left(b\right)
    = f\left(a\right)\circ \varphi_{f\left(a\right)}f\left(b\right)
    = f\left(a\right) + f\left(b\right),
\end{align*}
hence $f$ is a semi-brace homomorphism from $B$ to $C$. Therefore, the claim follows.
\end{proof}
\end{theor}
\noindent As a consequence of \cref{th:aff_semi-braces}, we also have that the affine structures associated to two isomorphic semi-braces are equivalent and  vice versa.
\medskip

Considering the class $\mathcal{O}b(\mathcal{SB})$ of semi-braces and the class $\mathcal{M}or(\mathcal{SB})$ of all homomorphisms between semi-braces, let us denote by $\mathcal{SB}$ the \emph{category of semi-braces}.
Moreover, if $\mathcal{O}b(\mathcal{A})$ is the class of all affine structures and $\mathcal{M}or(\mathcal{A})$ the class of all homomorphisms between affine structures, 
let us denote by $\mathcal{A}$ the \emph{category of affine structures}. 
\begin{cor}
The categories $\mathcal{SB}$ and $\mathcal{A}$ are equivalent.
\begin{proof}
By \cref{th:sk-aff} and \cref{th:aff-skew}, one can define the maps $F_1:\mathcal{O}b(\mathcal{SB})\to\mathcal{O}b(\mathcal{A})$ and $F_2:\mathcal{M}or(\mathcal{SB})\to\mathcal{M}or(\mathcal{A})$ such that $F_1\left(B\right) = \sigma^B$, for every $B\in\mathcal{O}b(\mathcal{SB})$, and $F_2\left(f\right) =f$, for every $f\in\mathcal{M}or(\mathcal{SB})$. Thus, by \cref{th:aff_semi-braces}, we have that $F:= \left(F_1, F_2\right)$ is a functor from the category $\mathcal{SB}$ to the category $\mathcal{A}$. 
Analogously, one can define the maps $G_1:\mathcal{O}b(\mathcal{A})\to\mathcal{O}b(\mathcal{SB})$ and $G_2:\mathcal{M}or(\mathcal{A})\to\mathcal{M}or(\mathcal{SB})$ such that $G_1\left(\sigma\right) = B_\sigma$, for every $\sigma\in\mathcal{O}b(\mathcal{A})$, and $G_2\left(f\right) =f$, for every $f\in\mathcal{M}or(\mathcal{A})$, so that 
$G:= \left(G_1, G_2\right)$ is a functor from the category $\mathcal{A}$ to the category $\mathcal{SB}$.
Furthermore, by \cref{rem:aff} and \cref{th:sk-aff},  it follows that
$FG = 1_{\mathcal{A}}$ and $GF = 1_{\mathcal{SB}}$, where $1_{\mathcal{A}}$ and $1_{\mathcal{SB}}$ are the identity functor on $\mathcal{A}$ and $\mathcal{SB}$, respectively. Therefore, the claim follows.
\end{proof}
\end{cor}
\medskip

Let us conclude this section by illustrating that the existence of a cancellative affine structure on a group $G$ is strictly linked to a Zappa-Sz\'{e}p product of $G$ and $G$ as semigroups. 
For this purpose and to fix the notation, let us recall the Zappa-Sz\'{e}p product of two semigroups (see \cite{Ku83} and \cite{Za42}).
Let $\left(S,\circ\right), \left(T,\circ\right)$ be two semigroups, and $\eta:T\to S^{S}$ and $\delta:S\to T^{T}$ maps. Set $^u a:= \eta\left(u\right)\left(a\right)$ and $u^a:= \delta\left(a\right)\left(u\right)$, for all $a\in S$ and $u\in T$, if the following conditions hold
\begin{align}
    \label{S1}
	&^{u}(a\circ b) = {}^u a \circ {}^{u^a}b	
	&^{u\circ v}a = {}^u\left(^va\right) \tag{Z1}\\
	\label{S2}
	&\left(u\circ v\right)^a = u^{^va}\circ v^a
	&u^{a\circ b}=(u^a)^b\tag{Z2}
\end{align}
for all $a,b \in S$ and $u,v \in T$, then $S \times T$ is a semigroup with respect to the operation defined by
\begin{align}\label{prod-semigruppo-match}
	\left(a,u\right)\circ \left(b,v\right)
	=\left(a\circ {}^{u}b, u^b\circ v\right),
\end{align}
for all $a,b \in S$ and $u,v \in T$, called the \emph{Zappa-Sz\'{e}p product of $S$ and $T$ via $\eta$ and $\delta$}.
%and is denoted by $S\bowtie T$. 
In addition, if $S$ and $T$ are groups 
%, ${}^u0 = 0$, $0^a = 0$, 
and  $\eta(T)\subseteq \Sym(S)$ and $\delta(S)\subseteq \Sym(T)$, then $S\times T$ endowed with the operation \eqref{prod-semigruppo-match} is a group.
\medskip

%The following is a characterization for a group $G$ to be endowed with a cancellative affine structure, or equivalently of a structure of a cancellative semi-brace. 
Below we show that, given a group $G$, specific
semigroup Zappa-Sz\'{e}p  products of $G$ and $G$ determine affine structures on $G$. Moreover, the converse is true if  we consider cancellative affine structures.
We remark that this result is intrinsic in \cite[Lemma 6.9]{SmVe18} and \cite[Theorem 5.3]{DeC19} in the context of skew braces, and in
\cite[Section 2]{Ru20-3} for  braces.
%(see also \cite[Theorem 2]{LuYZ00} and \cite[Remark 3.2]{GuVe17}).
\begin{prop}\label{prop:Zappa}
    Let $G = \left(B, \circ\right)$ be a group.
    If $G\bowtie G$ is the Zappa-Sz\'{e}p  product of $G$ and $G$ via $\eta:B\to B^B$ and $\delta:B\to B^B$ such that 
    \begin{align}\label{eq:prop-prod}
        u\circ a ={}^ua \circ {u}^{a},
    \end{align}
    for all $a\in S$ and $u\in T$, then the map $\sigma:B\to B^B$ defined by $\sigma_u\left(a\right) = {}^{u^-}a$, for every $u\in B$, is an affine structure on $G$.\\ 
    Conversely, if $\sigma$ is a cancellative affine structure on $G$, considered $\eta:B\to B^B$ and $\delta:B\to\Sym_B$ the maps defined by 
    ${}^ua:= \sigma_{u^-}\left(a\right)$ and ${u}^{a}:= \left({}^ua\right)^-\circ u\circ a$, for all $a,u\in B$, respectively,  then \eqref{eq:associative} holds and $G\times G$ endowed with the operation \eqref{prod-semigruppo-match} is the Zappa-Sz\'{e}p  product of $G$ and $G$ via $\eta$ and $\delta$. 
     
    \begin{proof}
    At first assume that $Z$ is the Zappa product of $G$ and $G$ via $\eta:B\to B^B$ and $\delta:B\to B^B$ such that \eqref{eq:prop-prod}  is satisfied.
    Then, if $a,b,u\in B$, %since $\left(u^-\right)^a\circ a^- = \sigma_u\left(a\right)^-\circ u^-$,
    we have that
    \begin{align*}
        \sigma_u\left(a\circ\sigma_a\left(b\right)\right)
        &= {}^{u^-}\left(a\circ { }^{a^-}b\right)
        = {}^{u^-}a\circ {}^{\left(u^-\right)^a\circ a^-}b
        = \sigma_u\left(a\right)\circ{}^{\sigma_u\left(a\right)^-\circ u^-}\left(b\right)\\ 
        &= \sigma_u\left(a\right)\circ\sigma_{\sigma_u\left(a\right)}\sigma_u\left(b\right),
    \end{align*}
    hence $\sigma$ is an affine structure on $G$.\\
    Conversely, suppose that $\sigma$ is a cancellative affine structure on $G$.  Then, \eqref{eq:prop-prod} is obvious. Besides, if $a,b,u,v\in B$, clearly
    $^{u\circ v}a ={}^{u}\left(^{v}a\right)$ and,  since $\sigma_{0} = \id_B$,
    \begin{align*}
      u^{a\circ b}
      &= \sigma_{u^-}\left(a\circ b\right)^-\circ u \circ a\circ b
      = \sigma_{\left( \sigma_{u^-}\left(a\right)^-\circ u\circ a \right)^-}\left(b\right)^-\circ \sigma_{u^-}\left(a\right)^-\circ u\circ a\circ b\\
      &= \left({}^{u^a}b\right)^-
      \circ
      \left({}^{u}a\right)^-\circ u\circ a\circ b
      = \left({}^{u^a}b\right)^-\circ u^{a}\circ b 
      = \left(u^a\right)^b.
    \end{align*}
    In addition, 
    %since $\sigma_{0} = \id_B$, we have that 
    \begin{align*}
        ^{u}\left(a\circ b\right)
        = \sigma_{u^-}\left(a\circ\sigma_a\sigma_{a^-}\left(b\right)\right)
        = \sigma_{u^-}\left(a\right)\circ
        \sigma_{\left(u^a\right)^-}\left(b\right)
        =\, ^{u} a \circ\, ^{u^a}b.
    \end{align*}
    and
    \begin{align*}
        u^{^va}\circ v^a
        = \left(\sigma_{u^-}\sigma_{v^-}\left(a\right)\right)^-\circ u
        \circ v\circ a
        = \sigma_{\left(u\circ v\right)^-}\left(a\right)^-\circ u\circ v\circ a
        = \left(u\circ v\right)^{a},
    \end{align*}
    thus \eqref{S1} and \eqref{S2} are satisfied.
    Therefore, the claim follows.
    \end{proof}
\end{prop}
\medskip

\begin{rem}
    Observe that the maps $\eta$ and $\delta$ involved in \cref{prop:Zappa} coincide with the known maps $\lambda$ and $\rho$ in \cref{prop:lambda-rho}, respectively. Indeed, by \cref{th:sk-aff}, we have that ${}^ua = \lambda_u\left(a\right)$ and 
    $u^a = \lambda_u\left(a\right)^-\circ u\circ a = \left(u^- + a\right)^-\circ a = \rho_a\left(u\right)$.
\end{rem}
\medskip

%*****************************************************
\section{The special cases of skew braces}
%*****************************************************

In this section, we analyze the skew braces obtained by the examples of groupal affine structures provided in \cref{Section-aff}.
\medskip

Initially, let us recall that,  according to \cite[Definition 2]{Ch19} (removing the assumption of finiteness),   a skew brace $B$ is said to be a \emph{bi-skew brace} if, other than  condition \eqref{eq:skew}, 
the identity
\begin{align}\label{cond-bi-skew-I}
    a + b\circ c 
    = \left(a + b\right)\circ a^-\circ \left(a + c\right)\tag{$\ast '$}
\end{align}
holds, for all $a,b,c\in B$. 
Equivalently, we observe that a skew brace $B$ is a bi-skew brace if the identity
\begin{align}\label{cond-bi-skew-II}
    a + b\circ c 
    = \left(a + b\right)\circ
    \left(a + \left(-a\right)\circ c\right)
    \tag{$\ast ''$}
\end{align}
holds, for all $a,b,c\in B$, that is the relation consistent with \eqref{eq:skew}. Easy examples of bi-skew braces are those obtained starting from a group $\left(B,\circ\right)$ and by setting $a + b:= a\circ b$ or  $a + b:= b\circ a$, for all $a,b\in B$, that, according to the terminology in \cite{KoTr20} and \cite{CaMaMiSt22} are called \emph{trivial skew brace} and \emph{almost trivial skew brace}, respectively.
\medskip

Below, we provide the following characterization that allows one for obtaining all bi-skew braces on a given group. In particular, we have that any bi-skew brace on a fixed group $G$ corresponds to specific semidirect products of $G$ with itself.
\begin{prop}\label{prop-bi-skew}
    Let $G = \left(B,\circ\right)$ be a group and $\sigma$ a groupal affine structure on $G$. Then, $B_{\sigma}$ is a bi-skew brace if and only if $\sigma^B_a$ is an automomorphism of $\left(B,\circ\right)$, for every $a\in B$.
    \begin{proof}
    Observe that $B_\sigma$ is a bi-skew brace if and only if the condition \eqref{cond-bi-skew-I} is satisfied.
    Since
    $a + b\circ c = a\circ\sigma^B_a\left(b\circ c\right)$ and
    \begin{align*}
        \left(a + b\right)\circ a^-\circ\left(a + c\right)
        %= a\circ\sigma^B_a\left(b\right)\circ a^-\circ a\circ\sigma^B_a\left(c\right)
        = a\circ\sigma^B_a\left(b\right)\circ\sigma^B_a\left(c\right),
    \end{align*}
    for all $a, b, c\in B$, the statement follows by comparing the two equalities obtained.
    \end{proof}
\end{prop}

\noindent In other words, as one can expect, a skew brace $B$ is a bi-skew brace if and only if $\lambda_a\in\Aut\left(B,\circ\right)$, for every $a\in B$.
In addition, if $\psi_B$ is the groupal affine associated to the skew brace $\left(B, \circ, +\right)$, then $\psi^{B}_{a} = \sigma^{B}_{a^-}$, for every $a\in B$. In fact, if $a,b\in B$, then 
$a\circ b = a + \psi^{B}_{a}\left(b\right)$ and so
$\psi^{B}_{a}\left(b\right) = \lambda_a\left(b\right)
= \sigma^{B}_{a^-}\left(b\right)$.
%In light of \cref{prop-bi-skew}, one can observe that fixed a group $G = \left(B,\circ\right)$ and an abelian affine structure $\sigma$ on $G$, then $B_{\sigma}$ is a bi-brace if and only if $\sigma^B_a$ is an automomorphism of $\left(B,\circ\right)$, for every $a\in B$.
\medskip

The following is another characterization of skew braces that are bi-skew braces. In particular, by \cref{th:sk-aff}, such a characterization is consistent with the results contained in \cite[Theorem 3.1]{Ca20}
and in \cite[Proposition 3.6]{BaNeYa22x}.
%where there compared bi-skew braces (named \emph{symmetric skew braces}) and other specific skew braces called \emph{$\lambda$-anti-homomorphic skew braces}.
\begin{cor}
Let $B$ be a skew brace. Then, $B$ is a bi-skew brace if and and only if 
%$\sigma^B_{a+b} = \sigma^B_{b\circ a}$, 
$\sigma_{a\circ\sigma_a\left(b\right)} = \sigma_{b\circ a}$,
for all $a,b\in B$, or, equivalently $\sigma^B$ is a homomorphism from $\left(B,+\right)$ to $\Aut\left(B,+\right)$.
\begin{proof}
Let $a,b,c\in B$. If $B$ is a bi-skew brace, by \cref{prop-bi-skew},
\begin{align*}
    \sigma^B_{a}\left(b\right)\circ \sigma^B_{\sigma^B_{a}\left(b\right)}\sigma^B_{a}\left(c\right)
    = \sigma^B_{a}\left(b\circ\sigma^B_{b}\left(c\right)\right)%&\mbox{by \eqref{eq:associative}}\\
    = 
    \sigma^B_{a}\left(b\right)\circ\sigma^B_{a}\sigma^B_{b}\left(c\right)
\end{align*}
i.e., $\sigma^B_{a\circ\sigma_a\left(b\right)} = \sigma^B_{b\circ a}$. %that implies $\lambda_{a+b} = \lambda_{b}\lambda_{a}$.
Conversely, if we assume that $\sigma^B_{a\circ\sigma_a\left(b\right)} = \sigma^B_{b\circ a}$, for all $a,b\in B$, we obtain that
\begin{align*}
    \sigma^B_{a}\left(b\right)\circ\sigma^B_{a}\left(c\right)
    &=  \sigma^B_{a}\left(b\right)\circ\sigma^B_{b\circ a}\sigma^B_{b^-}\left(c\right)
    = \sigma^B_{a}\left(b\right)\circ\sigma^B_{a+b}\sigma^B_{b^-}\left(c\right)\\
    &= \sigma^B_{a}\left(b\right)\circ\sigma^B_{\sigma^B_{a}\left(b\right)}\sigma^B_{a}\left(\sigma^B_{b^-}\left(c\right)\right)
    =\sigma^B_{a}\left(b\circ\sigma^B_b\sigma^B_{b^-}\left(c\right)\right)
    = \sigma^B_{a}\left(b\circ c\right),
\end{align*}
hence, by \cref{prop-bi-skew}, $B$ is a bi-skew brace. Therefore the claim is proved.
% As a consequence of \cref{prop-bi-skew} and \eqref{eq:associative}, we obtain that
% \begin{align*}
%     \sigma^B_{a+b} = \sigma^B_{b\circ a},
% \end{align*}
% for all $a,b\in B$, \textcolor{orange}{and vice-versa}.\\
% In particular, it follows that $\sigma^B_{a^-} = \sigma^B_{-a}$, for every $a\in B$, or, equivalently, $\lambda_{a^-} = \lambda_{-a}$, for every $a\in B$.
\end{proof}
\end{cor}

\noindent By the previous result, it follows also that any bi-skew brace satisfies $\sigma^B_{a^-} = \sigma^B_{-a}$, for every $a\in B$, or, equivalently, $\lambda_{a^-} = \lambda_{-a}$, for every $a\in B$.
\bigskip

In the next, we analyze skew braces obtained through the instances of groupal affine structures provided in \cref{Section-aff}.
\begin{ex}%\hspace{0.1mm}\\%\label{ex:f}
Given a group $G = \left(B, \circ\right)$, by \cref{prop-bi-skew} we obtain  that
the groupal affine structures $\sigma:B\to\Sym_B$  in \cref{exs:1} defined by $\sigma_a\left(b\right) = f\left(a\right)^-\circ b\circ f\left(a\right)$, for all $a,b\in B$, with $f$ an idempotent endomorphism of $G$,
give rise to bi-skew braces $B_\sigma$ having $G$ as multiplicative group.
Clearly, if $f\left(a\right) =0$, for every $a\in B$, then $B_{\sigma}$ is the trivial skew brace on $G$; if $f=\id_B$, then $B_{\sigma}$ is the almost trivial skew brace on $G$.\\
Among these examples of  bi-skew braces, one can find some %bi-skew braces 
that are not $\lambda$-homomorphic, i.e., such that $\lambda_{a+b}\neq\lambda_a\lambda_b$. In this way, by the observations in \cite[Section 5]{BaNeYa22x}, one can deduce that these are skew braces different from those provided in \cite{Ch19}, \cite{Ko21, Ko22}, and \cite{CarSt21, CarSt21x}. 
To show this fact, at first note that
%it is a routine computation to verify that 
$B_\sigma$ is $\lambda$-homomorphic if and only if $$f\left(a\circ b\right)\circ c\circ f\left(a\circ b\right)^- = f\left(b\circ a\right)\circ c\circ f\left(b\circ a\right)^-,$$for all $a,b\in B$.
Then, considered $C_2 = \langle x\rangle$ the cyclic group of order $2$ and $\Sym_3 = \langle \left(1\,2\,3\right), \left(1\,2\right)\rangle$ the symmetric group of degree $3$, let $G$ be the direct product $G = C_2\times \Sym_3$ (in particular $G$ is isomorphic to the dihedral group $D_{12}$) and $f:G\to G$ the idempotent endomorphism of $G$ given by 
$f\left(x,\left(1\,2\right)\right) = \left(x,\left(1\,2\right)\right)$,   $f\left(x,\id_3\right) = \left(1,\id_3\right)$,
and $f\left(1,\left(1\,2\,3\right)\right) = \left(1, \left(1\,2\, 3\right)\right)$.
By choosing $a = c = \left(x,\left(1\,2\right)\right)$ and $b = \left(1, \left(1\,2\, 3\right)\right)$, we obtain that  $f\left(a\circ b\right)\circ c\circ f\left(a\circ b\right)^-\neq f\left(b\circ a\right)\circ c\circ f\left(b\circ a\right)^-$. 
Consequently, the bi-skew brace $B_\sigma$ is not $\lambda$-homomorphic.
\end{ex}
\medskip

\begin{exs}\label{ex:esZ-Z2l}\hspace{1mm}
\begin{enumerate}
    \item Let $G =\left(B,\circ\right)$ be the infinite cyclic group
    %, $B:= \{g^k \ | \ k\in\mathbb{Z}\}$, 
    and  $\sigma:B\to\Sym_{B}$ the affine structure in \cref{ex:Z}-$1.$ defined by $\sigma_{g^k}\left(g^l\right):= g^{\left(-1\right)^{k}l}$, for all $k,l\in \mathbb{Z}$. Then, it is a routine computation to verify that $\sigma_{g^k}$ is an automorphism of $G$, for every $k\in\mathbb{Z}$, thus, by \cref{prop-bi-skew}, $B_{\sigma}$ is a bi-skew brace.
    More specifically, the addition of $B_{\sigma}$ is given by
    \begin{align}\label{eq:sum-infinite-dih}
        g^k \, +\, g^l = g^{k + \left(-1\right)^kl},
    \end{align}
    for all $k,l\in\mathbb{Z}$, hence $\left(B, +\right)$  is isomorphic to the infinite dihedral group $\mathbb{D}_{\infty}$. 
    Note also that the skew brace $\left(B,\circ,+\right)$ coincides with the structure of brace contained in \cite[Proposition 6]{Ru07-cyc}.
    %\textcolor{blue}{\cite[Proposition 6]{Ru07-cyc} and \cite[Example 2]{CCSt16}}.\\
    \item By similar observations in  $1.$, we        obtain that the first class of examples of affine structure on the cyclic group 
    %$G$ of order $m = 2l$ 
    $C_m$ with $m=2l$
    in \cref{ex:Z/Z2l} are bi-skew braces with additive group $\left(B, +\right)$ isomorphic to the dihedral group $\mathbb{D}_{l}$. In particular, in the case of $m=6$, the skew brace $\left(B,\circ, +\right)$ coincides with the brace of order $6$ obtained by Rump in \cite[Example 3]{Ru07}.
%with additive group $\left(B, +\right)$ isomorphic to the dihedral group $\mathbb{D}_{6}$ of order $6$, i.e., the symmetric group $\Sym_3$.
\item If $G =\left(B, \circ\right)$ is the infinite cyclic group 
%$B:= \{g^k \ | \ k\in\mathbb{Z}\}$  
and  $\sigma:B\to\Sym_{B}$ is the affine structure in \cref{ex:Z}-$2.$ defined by 
$\sigma_{g^k}\left(g^l\right):= g^{k\left(-1 + \left(-1\right)^l\right) + l}$, for all $k,l\in \mathbb{Z}$, then one can check that the maps $\sigma_{g^k}$ are not automorphisms of $G$, hence $B_\sigma$ is not a bi-skew brace. Furthermore, the sum of such a skew brace is given by
\begin{align*}
    g^k\, +\, g^l = g^{l + \left(-1\right)^{l}k},
\end{align*}
for all $k,l\in\mathbb{Z}$.
Besides, note that this structure of skew brace on $\left(\mathbb{Z}, +\right)$ is nothing but the opposite skew brace obtained in $1)$, i.e., the skew brace obtained by considering the opposite sum of such a skew brace (see \cite[Proposition 3.1]{KoTr20}).
%However, it identifies a class of skew brace that is not isomorphic to that in $1)$.
\item %Therefore, the skew brace in \cref{ex:queCSV} is not 
%\medskip
%... If $G =\left(B, \circ\right)$ is the cyclic group of $m$ elements $\langle\, g\, \rangle$ or ...then one can check that $\sigma_{g^k}$ is an automorphism of $G$ if and only if $m\mid 4$.
The affine structures on the cyclic group
%$G$ of order $m=2l$ 
$C_m$ with $m=2l$
in the second class of examples in \cref{ex:Z/Z2l} give rise to bi-skew braces if $m\mid 4$. In fact, it is a routine computation to check that $\sigma_{g^k}$ is an automorphism of $G$ if and only if $m\mid 4$.
\end{enumerate}
\end{exs}
\medskip

% %In the next, we provide some examples of skew braces obtained by means of groupal affine structures.
% \begin{exs}\label{ex:Z-skew}\hspace{1mm}
% \begin{enumerate}
     
%     %In particular, we obtain that $B$ is also a bi-skew brace.
%     %Clearly, $B$ is  not a brace since, for instance, $1 + 2 = 5\neq 3 = 2+1$.
%     % Moreover, reversing the role of $+$ and $\circ$ we obtain the structure of brace of order $6$ obtained by Rump in \cite[Example 3]{Ru07}. In particular, we obtain that $B$ is also a bi-skew brace.
    
%     \item If $G = \left(B, \circ\right)$ is the cyclic group of $m$ elements $\langle\, g\, \rangle$ or the infinite cyclic group and $\sigma:B\to\Sym_{B}$ is the affine structure in *** \cref{ex:Z}-$2.$*** defined by 
%     $\sigma_{g^k}\left(g^l\right):= g^{k\left(-1 + \left(-1\right)^l\right) + l}$, for all $k,l\in \mathbb{Z}$, then one can check that $\sigma_{g^k}$ is an automorphism of $G$ if and only if $m\mid 4$. Therefore, the skew brace in \cref{ex:queCSV} is not a bi-skew brace.
    
%     \item èèè
% \end{enumerate}
% \end{exs}

\begin{ex}
 The skew brace on the cyclic group $C_8$ 
 %of order $8$ 
 obtained through the affine structure in \cref{ex:C8+C8=D8} is a non-trivial brace with additive group isomorphic to $C_8$. 
 In particular, this brace is obtained through  the affine structure $\omega$ which gives rise to a skew brace that has the multiplicative structure isomorphic to $C_8$ and additive structure isomorphic to the dihedral group $\mathbb{D}_8$.
\end{ex}

\medskip

% \textcolor{red}{Therefore, it identifies a new class of skew brace on $\mathbb{Z}$ in the direction of \cite[Problem 2.27]{Ve19}. \footnote{\ Qui il gruppo additivo è il gruppo opposto del gruppo diedrale infinito quindi è anch'esso il gruppo diedrale infinito.}}

\begin{rem}
In \cite[Problem 2.27]{Ve19}, Vendramin posed the question of classifying isomorphism classes of skew braces with multiplicative group isomorphic to $\left(\mathbb{Z},+\right)$. 
In this regard,  the skew brace in \cref{ex:esZ-Z2l}-$3)$ identifies a class of skew brace that is different to that in \cref{ex:esZ-Z2l}-$2)$ which is already known. Note also that \cref{ex:esZ-Z2l}-$3)$ can be found in \cite{StTr22x} where the problem mentioned above is solved.
\end{rem}

\medskip

%*****************************************************
\section{Affine structures on Zappa products of groups}
%*****************************************************

This section aims to provide a construction of affine structures on groups that are the Zappa product of two given ones both endowed with affine structures. Furthermore, we show that such a construction allows to obtain semi-braces that are different from those determined by the matched product of semi-braces contained in \cite[Theorem 6]{CCSt20-2} (see also \cite[p. 249]{JeAr19}).

%Furthermore, we compare semi-braces obtained through such a construction with those obtained by the matched product of semi-braces contained in \cite[Theorem 6]{CCSt20-2} (see also \cite[p.249]{JeAr19}) and we show that, in general, they are different each other.
%

\medskip

To this purposes, let us first present a construction of groups that is equivalent to that of Zappa product of two groups (this one recalled at the end of \cref{section-semi-aff}) that is already intrinsic in \cite{Ba18}, \cite{SmVe18}, \cite{CCSt20}, \cite{CCSt20-2}, and \cite{CaMaSt21}. 
This alternative way to consider Zappa product of two groups allows one to compare the construction provided in this work with the matched product of semi-braces, skew braces, and braces contained in the papers mentioned above.

Let $H:=\left(S,\circ\right), K:= \left(T,\circ \right)$ be groups, $\alpha:T \to \Sym_S$ a homomorphism from $H$ to the symmetric group of $S$, and $\beta:S\to \Sym_T$ a homomorphism from $K$ to the symmetric group of $T$ such that
    \begin{align}\label{eq:mps}
     \alphaa{}{u}{\left(\alphaa{-1}{u}{\left(a\right)}\circ b\right)} 
     = a\circ\alphaa{}{\beta^{-1}_{a}{\left(u\right)}}{\left(b\right)}
     \qquad\quad
     \beta_{a}\left(\beta^{-1}_{a}\left(u\right)\circ v\right) = u\circ \beta_{\alphaa{-1}{u}{\left(a\right)}}\left(v\right)
    \end{align}
    are satisfied, for all $a,b \in S$ and $u,v \in T$. Let us call the quadruple $(S,T,\alpha,\beta)$ a \emph{matched product system of groups}. 
    % Then, the structure $G:=\left(S\times T, \circ\right)$ is a group where
    % \begin{align}\label{eq:prod-alfa_beta}
    %     \left(a,u\right)\circ\left(b,v\right)
    %     := \left(
    %     a\circ\alphaa{}{\beta^{-1}_{a}{\left(u\right)}}{\left(b\right)},\ u\circ \beta_{\alphaa{-1}{u}{\left(a\right)}}\left(v\right)
    %     \right),
    % \end{align}
    % for all $\left(a,u\right),\left(b,v\right)\in S\times T$.
    % Indeed, it is easy to verify that $G$ is a semigroup and since
    % $\alpha_u\left(0\right)
    %     = 0\circ\alpha_{\beta^{-1}_0\left(u\right)}\left(0\right)
    %     = \alpha_u\left(\alpha_u^{-1}\left(0\right)\circ 0\right)
    %     = 0$, and similarly $\beta_a\left(0\right) = 0$, then $G$ is a monoid with identity $\left(0,0\right)$. Moreover, any $\left(a,u\right)\in S\times T$ has inverse given by
    % $\left(\alpha^{-1}_{\beta^{-1}_a\left(u\right)}\left(a^-\right),\, \beta^{-1}_{\alpha^{-1}_u\left(a\right)}\left(u^-\right)\right)$. 
    Then, the binary operation given by
    \begin{align}\label{eq:prod-alfa_beta}
        \left(a,u\right)\circ\left(b,v\right)
        := \left(
        a\circ\alphaa{}{\beta^{-1}_{a}{\left(u\right)}}{\left(b\right)},\ u\circ \beta_{\alphaa{-1}{u}{\left(a\right)}}\left(v\right)
        \right),
    \end{align}
    for all $\left(a,u\right),\left(b,v\right)\in S\times T$, makes $S\times T$ into a group with identity $\left(0,0\right)$ and such that any $\left(a,u\right)\in S\times T$ has inverse given by
    $\left(\alpha^{-1}_{\beta^{-1}_a\left(u\right)}\left(a^-\right),\, \beta^{-1}_{\alpha^{-1}_u\left(a\right)}\left(u^-\right)\right) 
    = \left(\alpha^{-1}_u\left(a\right)^-,\; \beta^{-1}_a\left(u\right)^-\right)$.
    Finally, set $\sigma_u = \alpha_u$ and $\delta_a\left(u\right) = \beta^{-1}_{\alpha_{u}\left(a\right)}\left(u\right)$, it follows that $\left(S\times T, \cdot\right)$ is the Zappa product of $H$ and $K$ via $\sigma$ and $\delta$ and the map $f: S\times T \to S\times T$ defined by $\varphi\left(a,u\right) = \left(a, \beta^{-1}_{a}\left(u\right)\right)$, for all $a \in S$ and $u \in T$, is an isomorphism from $G$ to the Zappa product of $S$ and $T$. 
    Furthermore, let us denote the group $G$ as $H\bowtie_{\alpha,\beta} K$.
\medskip

Now, we recall the construction of the matched product of semi-braces. Let $S, T$ be semi-braces, $\alpha:T \to \Aut\left(S\right)$ a homomorphism from $\left(T,\circ\right)$ to the automorphism group of $\left(S,+\right)$, and $\beta:S\to \Aut\left(T\right)$ a homomorphism from $\left(S,\circ\right)$ to the automorphism group of $\left(T,+\right)$ such that identities in \eqref{eq:mps} are satisfied. Then, $(S,T,\alpha,\beta)$ is called a \emph{matched product system of semi-braces}.
In particular, the following conditions are satisfied
\begin{align}
\label{eq:mps-old}
\lambda_{a}\alpha_{\beta^{-1}_a\left(u\right)}
= \alpha_u\lambda_{\alpha^{-1}_u\left(a\right)}
\qquad
\lambda_{u}\beta_{\alpha^{-1}_u\left(a\right)}
= \beta_a\lambda_{\beta^{-1}_a\left(u\right)},
\end{align}
for every $\left(a,u\right)\in S\times T$.

\begin{theor}[Theorem 6, \cite{CCSt20-2}]\label{th:matched-old}
		Let $\left(S,T,\alpha,\beta\right)$ be a matched product system of semi-braces. Then, $S\times T$ with respect to the operations
		\begin{align*}
		\left(a,u\right) \oplus \left(b, v\right) 
		= \left(a + b, \, u + v\right)
		\qquad
		\left(a,u\right)\circ\left(b,v\right) 
		= \left(a\circ \alpha_{\beta^{-1}_a\left(u\right)}\left(b\right), 
		u\circ \beta_{\alpha^{-1}_u\left(a\right)}\left(v\right)
		\right),
		\end{align*}
	for all $\left(a,u\right), \left(b,v\right) \in S \times T$, is a left semi-brace called the \emph{matched product} of $S$ and $T$ via $\alpha$ and $\beta$. 
	%and denoted by $S\bowtie T$.
\end{theor}
\noindent Clearly, the additive structure is the direct sum of $\left(S,+\right)$ and $\left(T,+\right)$ and the multiplicative one is the Zappa product  $\left(S,\circ\right)\bowtie_{\alpha,\beta}\left(T,\circ\right)$.
\medskip
    
From now on, to avoid overloading the notation, if $\sigma^S$ and $\sigma^T$ are  affine structures on two given groups $H =\left(S,\circ\right)$ and $K = \left(T,\circ \right)$, respectively, we will write $\sigma^S_a$ and $\sigma^T_u$ simply as $\sigma_a$ and $\sigma_u$, respectively, for all $a\in S$ and $u\in T$. 
Moreover, we set $\bar{a}:= \alpha^{-1}_{u}\left(a\right)$ and $\bar{u}:= \beta^{-1}_{a}\left(u\right)$, for every pair $\left(a, u\right)\in S\times T$.
In addition, it is useful to recall that, if $B$ is a semi-brace, then the map $\rho_a$  can be written in terms of affine structure as $\rho_b\left(a\right) = \sigma_{a^-}\left(b\right)^-\circ a\circ b$, for all $a,b\in B$.
\medskip

The following theorem allows for constructing affine structures on groups $H\bowtie_{\alpha,\beta}K$. 
Specifically, it provides sufficient conditions so that the direct product $\sigma^S\times\sigma^T$ of two affine structures $\sigma^S$ and $\sigma^T$ is an affine structure on a group of the type $H\bowtie_{\alpha,\beta}K$. 
\begin{theor}\label{th:Zappa_prod}
    Let $H =\left(S,\circ\right), K = \left(T,\circ \right)$ be groups endowed with affine structures $\sigma^S$ and $\sigma^T$, respectively, and $(S,T,\alpha,\beta)$ a matched product system of groups. If the following conditions
        \begin{align}
        \label{eq:I}
        \sigma_0\alpha_{u} = \alpha_{u}\sigma_0
        &\qquad
        \sigma_0\beta_{a} = \beta_{a}\sigma_0
        \tag{I}\\
        %\end{align}
        %\begin{align}
        \label{eq:II}
            \sigma_{\alpha_u\left(a\right)} = \sigma_a
            &\qquad
            %\sigma_0\alpha_{\beta_b^{-1}\left(v\right)}\sigma_{\rho_b\left(a^-\right)}
            %= \sigma_{\rho_b\left(a^-\right)}\alpha_{\beta_{\sigma_a\left(b\right)}\sigma_u\left(v\right)}^{-1}
            \sigma_{\beta_a\left(u\right)} = \sigma_u
            \tag{II}\\
           %\end{align}
            %\begin{align}
            \label{eq:III}
            %\alpha_{\beta_b^{-1}\left(v\right)}\sigma_{\rho_b\left(a^-\right)}
            %= \sigma_{\rho_b\left(a^-\right)}\alpha_{\beta_{\sigma_a\left(b\right)}\sigma_u\left(v\right)}
            % \alpha_{\bar{v}}\sigma_{\rho_b\left(a^-\right)}
            % = \sigma_{\rho_b\left(a^-\right)}
            % \alpha_{\beta^{-1}_{\sigma_a\left(b\right)}\sigma_u\left(v\right)}
            \alpha_{\bar{v}}\sigma_{\rho_b\left(a^-\right)}
            = \sigma_{\rho_b\left(a^-\right)}
            \alpha_{\overline{V}}
            &\qquad
            %\sigma_0\beta_{\alpha_v^{-1}\left(b\right)}\sigma_{\rho_v\left(u^-\right)}
            %= \sigma_{\rho_v\left(u^-\right)}\beta_{\alpha_{\sigma_u\left(v\right)}\sigma_a\left(b\right)}^{-1}
            %\beta_{\alpha_v^{-1}\left(b\right)}\sigma_{\rho_v\left(u^-\right)}
            %= \sigma_{\rho_v\left(u^-\right)}\beta_{\alpha_{\sigma_u\left(v\right)}\sigma_a\left(b\right)}
            \beta_{\bar{b}}\sigma_{\rho_v\left(u^-\right)}
            = \sigma_{\rho_v\left(u^-\right)}
            \beta_{\overline{B}}
            \tag{III}
        \end{align}
        are satisfied, for all $a,b\in S$ and $u,v\in T$, where $\overline{V} = \beta^{-1}_{\sigma_a\left(b\right)}\sigma_u\left(v\right)$ and $\overline{B} = \alpha^{-1}_{\sigma_u\left(v\right)}\sigma_a\left(b\right)$,
        then the map $\sigma:= \sigma^S\times\sigma^T$
        is an affine structure on the group $H\bowtie_{\alpha,\beta}K$.
        % Moreover, if $\sigma^S$ and $\sigma^T$ are cancellative affine structures, then also $\sigma$ is, and if $\sigma^S$ and $\sigma^T$ are groupal affine structures, then also $\sigma$ is.
        \begin{proof}
        Let $\left(a,u\right), \left(b,v\right),\left(c,w\right)\in S\times T$. 
        Then, by \eqref{eq:II},
        \begin{align*}
        \sigma_{\left(a,u\right)\circ \left(b,v\right)}\left(c,w\right)
            &= \left(\sigma_{\alpha_{\bar{u}\left(b\right)}}\sigma_a\left(c\right),\,
            \sigma_{\beta_{\bar{a}\left(v\right)}}\sigma_u\left(w\right)\right)
            =\left(\sigma_{b}\sigma_a\left(c\right),\,
            \sigma_{v}\sigma_u\left(w\right)\right)\\
            &= \sigma_{\left(b,v\right)}\sigma_{\left(a,u\right)}\left(c,w\right),
        \end{align*}
        i.e., $\sigma$ is an anti-homomorphism.
        Moreover, set $X:= \sigma_{\left(a,u\right)}\left(\left(b,v\right)\circ \sigma_{\left(b,v\right)}\left(c,w\right)\right)$ we have that
        \begin{align*}
            X
            = 
            \left(\sigma_a\left(b\circ\alpha_{\bar{v}}\sigma_b\left(c\right)\right)
            , 
            \, \sigma_{u}\left(v\circ\beta_{\bar{b}}\sigma_v\left(w\right)\right)\right)
        \end{align*}
        where, the first component of $X$ can be written as 
        \begin{align*}
         \sigma_a\left(b\circ\alpha_{\bar{v}}\sigma_0\sigma_b\left(c\right)\right)
         &= \sigma_a\left(b\circ\sigma_{0}\alpha_{\bar{v}}\sigma_b\left(c\right)\right) &\mbox{ by \eqref{eq:I}}\\
         &= \sigma_a\left(b\circ\sigma_{b}\left(\sigma_{b^-}\alpha_{\bar{v}}\sigma_b\left(c\right)\right)\right)\\
         &= \sigma_a\left(b\right)\circ
        \sigma_{\sigma_a\left(b\right)}\sigma_a\left(\sigma_{b^-}\alpha_{\bar{v}}\sigma_b\left(c\right)\right)\\
       & = \sigma_a\left(b\right)\circ
        \sigma_{b^-\circ a\circ \sigma_a\left(b\right)}\alpha_{\bar{v}}\sigma_{b}\left(c\right)\\
        &= \sigma_a\left(b\right)\circ
        \sigma_{\rho_{b}\left(a^-\right)^-}\alpha_{\bar{v}}\sigma_{b}\left(c\right).
        \end{align*}
        In addition, set $Y:= \sigma_{\left(a,u\right)}\left(b,v\right)
          \circ\sigma_{\sigma_{\left(a,u\right)}\left(b,v\right)}\sigma_{\left(a,u\right)}\left(c,w\right)$, we obtain that
        \begin{align*}
          Y&= \left(\sigma_a\left(b\right), \sigma_u\left(v\right)\right)
          \circ\left(\sigma_{\sigma_a\left(b\right)}\sigma_a\left(c\right), \sigma_{\sigma_u\left(v\right)}\sigma_u\left(w\right)\right)\\
          &=\left(\sigma_a\left(b\right)\circ
          \alpha_{\overline{V}}\sigma_{\sigma_a\left(b\right)}\sigma_a\left(c\right),
          \, \sigma_u\left(v\right)\circ
          \beta_{\overline{B}}\sigma_{\sigma_u\left(v\right)}\sigma_u\left(w\right)
          \right)\\
          &=\left(\sigma_a\left(b\right)\circ
          \alpha_{\overline{V}}\sigma_{a\circ\sigma_a\left(b\right)}\left(c\right),
          \, \sigma_u\left(v\right)\circ
          \beta_{\overline{B}}\sigma_{u\circ \sigma_u\left(v\right)}\left(w\right)
          \right).
        \end{align*}
        Observing that
        \begin{align*}
            \sigma_{\rho_b\left(a^-\right)^-}
            \alpha_{\bar{v}}\sigma_{b}
            &=\sigma_{\rho_b\left(a^-\right)^-}
            \alpha_{\bar{v}}\sigma_0\sigma_{b}\\ &=\sigma_0\alpha_{\overline{V}}\sigma_{\rho_b\left(a^-\right)^-}\sigma_{b}&\mbox{by \eqref{eq:III}}\\
            &= \alpha_{\overline{V}}\sigma_{b\circ\rho_b\left(a^-\right)^- \circ \,0}&\mbox{by \eqref{eq:I}}\\
            &=\alpha_{\overline{V}}\sigma_{a\circ\sigma_a\left(b\right)},
        \end{align*}
       it follows that the first components of $X$ and $Y$, respectively, coincide. Similarly, the second components are equal. Therefore,  $\sigma$ is an affine structure on $H\bowtie_{\alpha,\beta}K$.
        \end{proof}
    \end{theor}
\medskip

Note that, if in \cref{th:Zappa_prod} the affine structures on the groups $H$ and $K$, respectively, are cancellative, then the identities in \eqref{eq:I} are trivially satisfied and the affine structure on $H\bowtie K_{\alpha,\beta}$ is cancellative. Moreover, if the affine structures on $H$ and $K$ are groupal, then also the affine structure $H\bowtie K_{\alpha,\beta}$ is. More specifically, we have the following characterization.
\begin{cor}
    Let $H =\left(S,\circ\right), K = \left(T,\circ \right)$ be groups endowed with cancellative affine structures and $(S,T,\alpha,\beta)$ a matched product system. Then, the map $\sigma:S\times T\to S\times T$ defined by $\sigma_{\left(a,u\right)}\left(b,v\right) = \left(\sigma_a\left(b\right),\, \sigma_u\left(v\right)\right)$, for all $\left(a,u\right), \left(a,u\right)\in S\times T$, is a cancellative affine structure on the group $H\bowtie_{\alpha,\beta} K$ if and only if
% \begin{align*}
% \sigma_{\alpha_u\left(a\right)} = \sigma_a
%     &\qquad
%     \sigma_{\beta_a\left(u\right)} = \sigma_u
%  \\
%     \alpha_{\bar{v}}\sigma_{\rho_b\left(a^-\right)}
%     = \sigma_{\rho_b\left(a^-\right)}\alpha_{\beta_{\sigma_a\left(b\right)}\sigma_u\left(v\right)}
%     &\qquad
%     \beta_{\bar{b}}\sigma_{\rho_v\left(u^-\right)}
%     = \sigma_{\rho_v\left(u^-\right)}\beta_{\alpha_{\sigma_u\left(v\right)}\sigma_a\left(b\right)}
% \end{align*}
the identities in \eqref{eq:II} and \eqref{eq:III}
are satisfied. 
\end{cor}
\medskip

%\noindent\textcolor{red}{*** Confronto con il matched product. -- Inserire anche i commenti relativi alla Definition $4$ e al Remark $5$ in \cite[p. 8]{CCSt20-2}. ***}

Now, note that if $B_\sigma$ is a semi-brace obtained through the affine structure constructed by \cref{th:Zappa_prod}, then
% \begin{align*}
%     \left(a,u\right) + \left(b,v\right)
%     = \left(
%     a\circ\alpha_{\beta^{-1}_a\left(u\right)}\sigma_a\left(b\right),\,
%     u\circ\beta_{\alpha^{-1}_u\left(a\right)}\sigma_u\left(v\right)
%     \right),
% \end{align*}
\begin{align*}
    \left(a,u\right) + \left(b,v\right)
    = \left(
    a\circ\alpha_{\bar{u}}\sigma_a\left(b\right),\,
    u\circ\beta_{\bar{a}}\sigma_u\left(v\right)
    \right),
\end{align*}
for all $\left(a,u\right), \left(b,v\right)\in S\times T$.
In general, it is obvious that such a sum is not the direct sum of the additive structure of $H$ and $K$ as it is in \cref{th:matched-old}.\\
We also have that the restrictions of $+$ and $\circ$ to the set
$S^{\ast} = \{\left(a,0\right) \ | \ a\in S\}$ are exactly component-wise. 
Indeed, this is a consequence of the fact that $\sigma_0\left(0\right) = 0$,
%(as shown in \cref{Section-aff}) 
$\alpha_u\left(0\right)
= 0\circ\alpha_{\beta^{-1}_0\left(u\right)}\left(0\right)     = \alpha_u\left(\alpha_u^{-1}\left(0\right)\circ 0\right)
        = 0$,  and, similarly, $\beta_a\left(0\right)=0$.
Analogously, the restriction of $+$ and $\circ$ to the set
$T^{\ast} = \{\left(0,u\right) \ | \ u\in T\}$ are exactly component-wise.
\medskip

\begin{rems}\label{rem:confronto}
Let $H:=\left(S,\circ\right)$ and $K:= \left(T,\circ \right)$ be groups, $\sigma^T$ and $\sigma^S$ cancellative affine structures on $H$ and $K$, respectively.
\begin{enumerate}
    \item If $\alpha:T \to \Sym_S$ and $\beta:S\to \Sym_T$ are group homomorphisms such that the identities \eqref{eq:mps} hold, then the conditions in \eqref{eq:mps-old} hold if and only if $\alpha_u\in\Aut\left(S,+\right)$ and $\beta\in \Aut\left(T, +\right)$. Indeed, if the conditions in \eqref{eq:mps-old} hold, since $\alpha_u\left(a\right)^- = \alpha_{\beta^{-1}_{\alpha_u\left(a\right)}\left(u\right)}\left(a^-\right)$, it follows that
\begin{align*}
    \alpha_{\beta^{-1}_{\alpha_u\left(a\right)}\left(u\right)}\lambda_{a^-}
    = \alpha_{\beta^{-1}_{\alpha_u\left(a\right)}\left(u\right)}\lambda_{\alpha^{-1}_{\beta^{-1}_{\alpha_u\left(a\right)}\left(u\right)}\alpha_u\left(a\right)^-}
    = \lambda_{\alpha_u\left(a\right)^- }
        \alpha_{\beta^{-1}_{\alpha_u\left(a\right)^- }\beta^{-1}_{\alpha_u\left(a\right)}\left(u\right)}
    = \lambda_{\alpha_u\left(a\right)^- }\alpha_u.
    \end{align*}
   Hence, we obtain that 
   \begin{align*}
       \alpha_u\left(a + b\right)
       &= \alpha_u\left(a\circ\lambda_{a^-}\left(b\right)\right)
       = \alpha_u\left(a\right)\circ\alpha_{\beta^{-1}_{\alpha_u\left(a\right)}}\lambda_{a^-}\left(b\right)
       = \alpha_u\left(a\right)\circ\lambda_{\alpha_u\left(a\right)^- }\alpha_u\left(b\right)\\
       &= \alpha_u\left(a\right) + \alpha_u\left(b\right),
   \end{align*}
   i.e., $\alpha_u\in\Aut\left(S,+\right)$. Similarly, one has that $\beta\in \Aut\left(T, +\right)$.
    The converse part is contained in \cite[Remark 5]{CCSt20-2}.
    \item If $\alpha:T \to \Sym_S$ and $\beta:S\to \Sym_T$ are group homomorphisms, if $\alpha_u\in\Aut\left(S,+\right)$ and $\beta_a\in \Aut\left(T, +\right)$, and $\lambda_a\alpha_{\bar{u}} = \alpha_u\lambda_{\bar{a}}$ and $\lambda_u\beta_{\bar{a}} = \alpha_a\lambda_{\bar{u}}$, for all $a\in S$ and $u\in T$, then the identities \eqref{eq:mps} are satisfied.
\end{enumerate}
\end{rems}
\medskip

Now, let us note that if $\left(S,T,\alpha,\beta\right)$ is a matched product system of semi-braces, then the affine structure associated to the group $\left(S\times T,\circ\right)$, by \cref{th:matched-old}, is the map $\bar{\sigma}:T\times S\to T\times S$ given by
\begin{align*}
    \bar{\sigma}_{\left(a,u\right)}\left(b,v\right) 
    = \left(\alpha^{-1}_{\bar{u}}\sigma_a\left(b\right), \, \beta^{-1}_{\bar{a}}\sigma_u\left(v\right)
    \right),
\end{align*}
for all $\left(a,u\right), \left(b,v\right)\in S\times T$. 
As shown by the examples below, one can see that,  in general, $\bar{\sigma}$ and $\sigma^S\times\sigma^T$ are not equivalent. 

\begin{ex}
Let $H = \left(S, \circ\right)$ be a group and $\sigma$ the affine structure on $H$ given by $\sigma_a =  f$, for every $a\in S$, with $f$ an idempotent endomorphism of $H$. Consider $T = S$, $\sigma_u =f$, for every $u\in T$, and  $\alpha:T\to \Sym_S$ and  $\beta: S\to \Sym_T$ the maps  given by $\alpha_u = \id_S$ and $\beta_{a}\left(u\right) = f\left(a\right)\circ u\circ f\left(a\right)^-$, for all $a\in S$ and $u \in T$, respectively. Then $\left(S,T, \alpha, \beta\right)$ is a matched product system of groups that satisfies the assumptions of \cref{th:Zappa_prod}. Then, the affine structure $\sigma_{\left(a,u\right)} = \left(f\left(b\right), f\left(v\right)\right)$ on the group $G= H\bowtie_{\alpha,\beta} K$ gives rise to the semi-brace having $G$ as multiplicative structure, namely, 
$$\left(a,u\right)\circ \left(b,v\right)
    = \left(a\circ b, \, u\circ f\left(a\right)\circ v\circ f\left(a\right)^-\right),$$
and such that the sum is given by
\begin{align*}
    \left(a,u\right) + \left(b,v\right)
    = \left(a\circ b, \, u\circ f\left(a\circ v\circ a^-\right)\right),
\end{align*}
for all $\left(a,u\right), \left(b,v\right)\in S\times T$.
It is a routine computation to check that $\alpha$ and $\beta$ satisfy also assumption of \cref{th:matched-old}, and that the sum $\oplus$ in \cref{th:matched-old} coincides with that in \cref{th:Zappa_prod} if and only if the group $\im f$ is abelian.
\end{ex}
\medskip

The following example ensures that, in general, a semi-brace obtained by \cref{th:Zappa_prod} is not isomorphic to that obtained in \cref{th:matched-old}.
%\textcolor{red}{In generale, un semi-brace ottenuto come il "vecchio" matched product di due semi-braces con somma commutativa dà luogo ad uno nuovo con somma ancora commutativa; il "nuovo", invece, dà luogo a un nuovo semibrace che in generale non è commutativo. Inoltre le $\alpha$ e $\beta$, per le osservazioni che mancano possono soddisfare entrambi i teoremi dando luogo a semi-brace non isomorfi ma aventi gruppo moltiplicativo isomorfo.}
\begin{ex}
Let $H=\left(S,\circ\right)$ be the cyclic group $C_6 = \langle a\rangle$ 
%of order $6$ 
and $K=\left(T,\circ\right)$ the cyclic group $C_2 = \langle u\rangle$.
%of order $2$. 
Let $\sigma^S$ be the groupal affine structure on the group $H$ given by $\sigma_{a^k}\left(a^l\right) = a^{\left(-1\right)^k l}$, for all $k,l\in\mathbb{Z}$, and $\sigma^T$ the abelian affine structure on the group $K$ given by $\sigma_{u^t} = \id_T$, for every $t\in\mathbb{Z}$. Consider $\alpha:T\to \Sym_S$ the map defined by $\alpha_{u^t}\left(a^k\right) = a^{\left(-1\right)^t k}$, for all $t,k\in\mathbb{Z}$, and $\beta:S\to \Sym_T$ the map such that $\beta_{a^k} = \id_T$, for every $k\in\mathbb{Z}$. 
Then, it is a routine computation to verify that $\left(S,T,\alpha,\beta\right)$ is a matched product system of groups that satisfies the hypotheses of \cref{th:Zappa_prod}. 
In particular, the groupal affine structure obtained in this way, namely, $\sigma_{\left(a^k,u^t\right)}\left(a^l,u^s\right) = \left(a^{\left(-1\right)^kl},\, u^t\right)$, determines a structure of skew brace with
\begin{align*}
    \left(a^k, u^t\right) + \left(a^l, u^s\right)
    = \left(a^{k + \left(-1\right)^{t+k} l},\, u^{t+s}\right)
    \qquad
    \left(a^k,u^t\right)\circ \left(a^l,u^s\right)
    = \left(a^{k +\left(-1\right)^t l},\, u^{t + s}\right),
\end{align*}
for all $k,l,t,s\in\mathbb{Z}$. In particular, the multiplicative group of  $B_{\sigma}$ is isomorphic to $\mathbb{D}_{12}$. Moreover, observe that if $k = 0$, $l = t = 1$, and $l = 2$, then $a^{k + \left(-1\right)^{t+k} h} = a^{-2} = a^{4}$, instead
$a^{l + \left(-1\right)^{s + l} k} = a^2$, hence the additive group of $B_{\sigma}$ is not abelian.
Finally, by $1.$ in \cref{rem:confronto}, the maps $\alpha$ and $\beta$ satisfy the hypotheses of \cref{th:matched-old}, hence we obtain a skew brace with multiplicative group equal to that of $B_\sigma$ and abelian additive structure that is the direct product of $C_6$ and $C_2$. 
\end{ex}
% \medskip
%
% \begin{ex}
% \textcolor{red}{Trovare un esempio in cui $\alpha_u$ e/o $\beta_a$ non sono automorfismi rispetto alla somma.}
% \end{ex}

\bigskip

\bibliography{bibliography}

\end{document}